\pdfoutput=1
\documentclass[10pt, a4paper]{article}

\usepackage{ifluatex} %
\usepackage[perpage]{footmisc}
\usepackage[a4paper,marginparwidth=2cm, scale=0.68, heightrounded]{geometry}
\usepackage[utf8]{inputenc} %
\usepackage{scrextend}
\usepackage[ngerman,main=english]{babel}
\babeltags{german=ngerman}

\usepackage{graphicx}
\usepackage[labelfont=bf]{caption}
\usepackage{subcaption}
\usepackage{amsmath,mathtools,amssymb,amsfonts, multicol}
\usepackage{tikz-cd} 
\usepackage{marvosym}
\usepackage{array}
\renewcommand{\arraystretch}{2}
\usepackage{slashed} %

\ifluatex
   \usepackage{fontspec}
   \usepackage{unicode-math} %
\else
  \usepackage{lmodern}
  \usepackage[T1]{fontenc}
\fi

\usepackage[autostyle=true]{csquotes} %
\usepackage[final]{microtype}

\usepackage{amsthm}
\usepackage{thmtools}
\usepackage[pdfusetitle]{hyperref}
\usepackage[citestyle=alphabetic,
            bibstyle=alphabetic,
            backend=biber,
            maxnames=5,
            url=true,
            urldate=comp,
            doi=true,
            isbn=false,
            giveninits=true,
            eprint=false,
            noerroretextools]{biblatex}

\bibliography{literatur.bib}
\AtEveryBibitem{\clearfield{month}}
\AtEveryBibitem{\clearfield{day}}
\AtEveryBibitem{\clearlist{language}}
\usepackage[shortlabels]{enumitem}
\setlist[itemize]{noitemsep}
\setlist[enumerate]{label=\upshape(\arabic*)}
\newlist{myenumi}{enumerate}{1}
\setlist[myenumi,1]{label=\upshape(\roman*)}
\newlist{myenuma}{enumerate}{1}
\setlist[myenuma,1]{label=\upshape(\alph*)}

\usepackage{xcolor}

\usepackage{todonotes}
\usepackage{xfrac}  
\usepackage{titlesec}
\titleformat*{\section}{\large\bfseries}
\titleformat*{\subsection}{\bfseries}
\declaretheorem[name=Theorem,numberwithin=section]{thm}
\declaretheorem[name=Theorem, numbered=no]{thm*}
\declaretheorem[name=Lemma,numberlike=thm]{lem}
\declaretheorem[name=Lemma,numbered=no]{lem*}
\declaretheorem[name=Claim,numbered=no]{claim*}
\declaretheorem[name=Corollary,numberlike=thm]{cor}
\declaretheorem[name=Proposition,numberlike=thm]{prop}
\declaretheorem[name=Definition,numberlike=thm, style=definition]{defi}
\declaretheorem[name=Example, numberlike=thm, style=remark]{ex}
\declaretheorem[name=Remark, numberlike=thm, style=remark]{rem}

\declaretheorem[name=Notation,numberlike=thm, style=remark]{nota}
\declaretheorem[name=Notation,numbered=no, style=remark]{nota*}
\declaretheorem[name=Construction,numberlike=thm, style=remark]{const}
\declaretheorem[name=Construction,numbered=no, style=remark]{const*}

\declaretheorem[name=Theorem]{prethm}

\declaretheorem[name=Lemma,numberlike=prethm]{prelem}
\declaretheorem[name=Corollary,numberlike=prethm]{precor}

\numberwithin{equation}{section}
\allowdisplaybreaks[4]

\usepackage[noabbrev]{cleveref} %
\crefname{figure}{Figure}{Figures}
\crefname{nota}{Notation}{Notations}
\crefname{table}{Table}{Tables}
\crefname{thm}{Theorem}{Theorems}
\crefname{lem}{Lemma}{Lemmas}
\crefname{defi}{Definition}{Definitions}
\crefname{cor}{Corollary}{Corollaries}
\crefname{prop}{Proposition}{Propositions}
\crefname{ex}{Example}{Examples}
\crefname{rem}{Remark}{Remarks}
\crefname{const}{Construction}{Constructions}
\crefname{conj}{Conjecture}{Conjectures}
\crefname{section}{Section}{Sections}
\crefname{chapter}{Chapter}{Chapters}
\crefname{appendix}{Appendix}{Appendices}
\crefdefaultlabelformat{#2\textup{#1}#3}
\creflabelformat{enumi}{(#2#1#3)}
\crefname{prethm}{Theorem}{}
\crefname{precor}{Corollary}{}
\crefname{prelem}{Lemma}{}
\crefname{pregen}{Principle}{}
\crefname{guide}{Guideline}{}

\usepackage{xparse}
\usepackage{autonum} 

\usepackage{thomasmacros}
\usepackage{titling}

\usepackage{authblk}


\title{\Large\textbf{Scalar curvature rigidity and the higher mapping degree}}
\author{Thomas Tony
}
\definecolor{Mydarkblue}{rgb}{0.0, 0.2, 0.4} 
\definecolor{myCustomColor}{rgb}{0.0, 0.2, 0.4}
\hypersetup{
  colorlinks=true,
  allcolors=blue,
  pdfauthor=Thomas Tony,
}
\date{\vspace{-2\baselineskip}}
\newcommand{\SE}{\SpinBdl_E}
\newcommand{\SM}{\SpinBdl M}
\newcommand{\SN}{\SpinBdl N}
\newcommand{\DE}{\Dirac_E}
\newcommand{\DL}{\Dirac_{\mathcal{L}}}
\newcommand{\ue}{u_\epsilon}
\newcommand{\ve}{v_\epsilon}
\DeclareMathOperator{\feblanc}{f_\epsilon}
\DeclareMathOperator{\hideg}{deg_{hi}}
\DeclareMathOperator{\Adeg}{\deg_{\hat{A}}}
\newcommand{\fe}[1]{\feblanc \brackets{#1}}
\NewDocumentCommand{\tensgr}{}{\mathbin{\widehat{\otimes}}}
\NewDocumentCommand{\boxgr}{}{\mathbin{\widehat{\boxtimes}}}
\NewDocumentCommand{\tens}{}{\otimes}
\AtEndDocument{\bigskip{\footnotesize%
  \textsc{Mathematisches Institut, Universität Münster, Einsteinstr. 62, 48149 Münster,
  Germany} \par  
  email:~\href{mailto:math@ttony.eu}{math@ttony.eu}, url:~\href{https://www.ttony.eu}{www.ttony.eu}
}}

\newcommand*{\keywordbullet}{\,\textcolor{darkgray}{\scriptsize\textbullet\footnotesize\,}}
\providecommand{\keywords}[4]
{
  \footnotesize
  \textbf{\textit{Keywords---}} #1\keywordbullet #2\keywordbullet #3\keywordbullet #4
}

\begin{document}
\maketitle
\begin{abstract}
  A closed connected oriented Riemannian manifold~$N$ with non-vanishing Euler characteristic, non-negative curvature operator and $0< 2\Ric_N<\scal_N$ is area-rigid in the sense that any area non-increasing spin map $f\colon M\to N$ 
  with non-vanishing $\hat{A}$-degree and $\scal_M\geq \scal_N \op f$ is a Riemannian submersion with $\scal_M=\scal_N \op f$. This is due to Goette and Semmelmann and generalizes a result by Llarull. In this article, we show area-rigidity for not necessarily orientable manifolds with respect to a larger class of maps $f\colon M\to N$ by replacing the topological condition on the $\hat{A}$-degree by a less restrictive condition involving the so-called higher mapping degree. This includes fiber bundles over even dimensional spheres with enlargeable fibers, e.g.~$\pr_1\colon S^{2n}\times T^k \to S^{2n}$. We develop a technique to extract from a non-vanishing higher index a geometrically useful family of almost $\Dirac$-harmonic sections. This also leads to a new proof of the fact that any closed connected spin manifold with non-negative scalar curvature and non-trivial Rosenberg index is Ricci flat.
\end{abstract}
\keywords{scalar curvature rigidity}{comparison geometry}{Dirac operator}{higher index theory}
\normalsize
\section{Introduction and main results}
A classical result in scalar curvature comparison geometry by \textcite{Llarull1998} states that a smooth area non-increasing map $f\colon M\to S^n$ from an $n$-dimensional closed connected Riemannian spin manifold of non-zero degree onto the round sphere, $n\geq 2$, with $\scal_M\geq n\brackets{n-1}=\scal_{S^n}$, is necessarily an isometry. An important consequence is that the round sphere is area-extremal in the sense that one cannot increase the metric and the scalar curvature simultaneously. \textcite{Goette2000} generalized this result to area non-increasing spin maps $f\colon M\to N$ of non-zero $\hat{A}$-degree onto a closed connected oriented Riemannian manifold of non-vanishing Euler characteristic and non-negative curvature operator. Further generalizations are by \textcite{Lott2021} to manifolds with boundary, by \textcite{Cecchini2022} to Lipschitz maps with lower regularity, and by \textcite{Bettiol2024}, who weakened the condition on the curvature operator of the target manifold in dimension four. \medbreak
In this article, we generalize the extremality and rigidity statement of \textcite[Theorem 2.4]{Goette2000} to spin maps between non-orientable manifolds and replace the topological condition on the $\hat{A}$-degree by a less restrictive index-theoretical condition involving the so-called higher mapping degree (see \cref{HigherMappingDegree}). Here spin map means that the map is compatible with the first and second Stiefel-Whitney classes of the involved manifolds. 
\begin{prethm}[\cref{RigidityThmHilbertBundleE} and \cref{RigidityThmHigherDergeeRosenberg}] \label{ThmA}
  Let $f\colon \brackets{M,g}\to \brackets{N,\overline{g}}$ be an area non-increasing spin map between two closed connected Riemannian manifolds of dimension~$n+k$ and~$n$, respectively. Suppose that the curvature operator of~$N$ is non-negative and
  \begin{equation}
    \chi\brackets{N}\cdot \hideg\brackets{f}\neq 0 \in \KO_k \brackets{\Cstar \pi_1\brackets{M}}.
  \end{equation}
  Then $\scal_M\geq \scal_N\op f$ on $M$ implies $\scal_M=\scal_N \op f$. If, moreover, $\scal_N>2\Ric_N>0$ (or~$f$ is distance non-increasing and $\Ric_N>0$), then $\scal_M\geq \scal_N\op f$ implies that~$f$ is a Riemannian submersion. 
\end{prethm}
\cref{ThmA} applies to the projections $\pr_1\colon S^{2n}\times T^k\to S^{2n}$ and $\pr_1\colon \RP^{2n}\times \Sigma^{8k+j}\to \RP^{2n}$ for $j\in \set{1,2}$ and $\Sigma^{8k+j}$ an exotic sphere with non-vanishing Hitchin invariant (\cref{ExampleTorus} and \cref{ExExoticSpehere}). For both projections the $\hat{A}$-degree vanishes, hence \cref{ThmA} is a proper extension of \cite[Theorem 2.4]{Goette2000}. The second example only works because of the generalization to non-orientable manifolds (see \cref{RemarkOrientable}). \medbreak 
The motivation for \cref{ThmA} is the fact that the Rosenberg index of a closed spin manifold~$M$ is the most general known index-theoretical obstruction to the existence of a positive scalar curvature metric on~$M$. Since the Rosenberg index of the $n$-torus does not vanish, the $n$-torus does not admit a positive scalar curvature metric. This information cannot be read off the classical index of the spin Dirac operator, hence the higher index carries more information. In particular, the classical rigidity statement by \textcite[][Theorem 2.4]{Goette2000} does not apply to $\pr_1\colon S^{2n}\times T^k\to S^{2n}$, but the higher version in \cref{ThmA} does. A further obstruction to a positive scalar curvature metric on a closed spin manifold is enlargeability \cite[][Theorem A]{Gromov1980}. It also yields that the $n$-torus does not admit a positive scalar curvature metric. In general, the enlargeability obstruction for a closed spin manifold is dominated by the Rosenberg obstruction \cite{Hanke2006, Hanke2007}, so \cref{ThmA} also applies to certain maps with enlargeable fibers (see \cref{PreCor} and \cref{ExampleProduct}). Alternatively, the methods used in the proof of \cref{ThmA} would carry over to almost flat bundles as they are used in \cite{Hanke2007}. \medbreak
\cref{ThmA} yields a rigidity statement for fiber bundles over 2-connected Riemannian manifolds of non-zero Euler characteristic and non-negative curvature operator. A classification of these manifolds (\cref{LemClassificationN}) yields the following rigidity statement:
\begin{precor}[\cref{CorRigidityFiberBundles}] \label{PreCor}
  Let~$N$ be one of the following classes of Riemannian manifolds:
  \begin{itemize}
    \item Sphere~$S^{2n}$, $n\geq 2$, equipped with a metric of non-negative curvature operator and $\scal_{S^n}>2\Ric_{S^n}>0$. 
    \item A quaternionic Grassmannian or the Cayley projective plane equipped with their symmetric Riemannian metrics.
  \end{itemize}
  Let $f\colon M\to N$ be a fiber bundle, where the total space~$M$ is a closed connected spin manifold and the typical fiber~$F$ satisfies 
  \begin{equation}
    \chi\brackets{N} \cdot \alpha\brackets{F} \neq 0 \in \KO_{\dim \brackets{F}}\brackets{\Cstar \pi_1\brackets[\big]{F}} \quad \textit{(e.g. } F \textit{ is enlargeable)}.
  \end{equation}
  Then the map~$f$ is a Riemannian submersion with $\scal_M=\scal_N \op f$ if~$f$ is area non-increasing and the scalar curvatures satisfy $\scal_M\geq \scal_N\op f$. Here~$\chi\brackets{N}$ denotes the Euler characteristic of~$N$, and $\alpha\brackets{F}$ denotes the Rosenberg index of the fiber. \medbreak
\end{precor} 
\subsection{Outline of the proof of \cref{ThmA}}
Let us first analyze the proof of the classical theorem of \textcite[Theorem 2.4]{Goette2000}. The key ingredient is the twisted spinor bundle $\SpinBdl \coloneqq \SpinBdl M\otimes f^* \SpinBdl N\to M$, which contains all necessary information about the map~$f$ and the manifolds~$M$ and~$N$. The proof divides into three steps. At first, the estimate  
\begin{equation}
  \Dirac^2 \geq \nabla^* \nabla + \tfrac{1}{4}\brackets{\scal_M-\scal_N\op f}
\end{equation}
holds by using the Schrödinger--Lichnerowicz formula together with the non-negativity of the curvature operator of~$N$ and the fact that~$f$ is area non-increasing. Second, the Atiyah-Singer index theorem yields the index formula 
\begin{equation} 
  \ind \Dirac = \chi\brackets{N}\cdot \Adeg \brackets{f}, \quad \text{with}
  \Adeg\brackets{f}\coloneqq \brackets[\big]{\hat{A}\brackets{M}f^*\omega} \sqbrackets{M}.
\end{equation}
Here $\omega$ is the fundamental class of~$N$ in cohomology and~$\sqbrackets{M}$ the fundamental class of~$M$. Third, the non-vanishing of the Euler characteristic of~$N$ and the $\hat{A}$-degree gives, combined with the index formula, a non-trivial $\Dirac$-harmonic section~$u$. Inserting~$u$ into the estimate for~$\Dirac^2$ gives that~$u$ is even parallel. Using again the estimate for~$\Dirac^2$ yields the extremality and rigidity statement.\medbreak
Let us start the outline of the proof of \cref{ThmA}. The spin map~$f\colon M\to N$ gives rise to a graded Real $\ComplexCl_{n+k,n}$-linear Dirac bundle $\SpinBdl M\tensgr f^*\SpinBdl N$ (see \cref{AlinearDiracOperatorsSubsection}). This is the spinor bundle associated to the vector bundle $TM\oplus f^*TN$ equipped with the indefinite metric $g\oplus \brackets{-f^*\overline{g}}$ (see \cref{SpinStructureSubsection}). This vector bundle is isomorphic as a Dirac bundle to some copies of the bundle~$\SpinBdl$ used by \textcite{Goette2000}, but carries additionally the $\ComplexCl_{n+k,n}$-linear structure. Motivated by the definition of the Rosenberg index, we perform an additional twist with the Mishchenko bundle of~$M$ (see \cref{MishchenkoBundleSubsection}) 
\begin{equation}
  \mathcal{L}\coloneqq \widetilde{M}\times_{\pi_1\brackets{M}} \Cstar \pi_1\brackets{M} \to M. 
\end{equation}
Here $\widetilde{M}$ is the universal cover of~$M$ and $\Cstar \pi_1\brackets{M}$ the maximal group $\Cstar$-algebra of the fundamental group of~$M$. The Mishchenko bundle is equipped with a canonical flat connection. This twisted bundle carries the structure of a graded Real $\ComplexCl_{n+k,n} \tensgr\Cstar \pi_1\brackets{M}$-linear Dirac bundle with induced Dirac operator $\DL$. The classical estimates for $\Dirac^2$ generalize to 
\begin{equation}
  \DL^2 \geq \nabla^* \nabla + \tfrac{1}{4}\brackets{\scal_M-\scal_N\op f}.
\end{equation}
A cut-and-paste principle for the higher index, combined with the Poincaré--Hopf lemma, leads to an analogous index formula for~$\DL$ as for~$\Dirac$ in the classical proof.
\begin{prethm}[\cref{IndexThm} and \cref{HigherMappingDegree}] \label{ThmB}
  The higher degree of a spin map~$f\colon M\to N$ is defined as the higher index of the Dirac operator of the spinor bundle of the fiber~$M_p$ over a regular value~$p$ of~$f$ twisted by~$\mathcal{L}$ restricted to~$M_p$. Then the following holds
  \begin{equation}
    \ind \DL = \chi\brackets{N} \cdot \hideg \brackets{f} \in \KO_k \brackets{\Cstar \pi_1\brackets{M}}.
  \end{equation}
\end{prethm}
A non-vanishing higher index of an $\A$-linear Dirac operator~$\Dirac$ does not necessarily give rise to a non-trivial element in its kernel. The functional calculus can be used (\cref{DiracInvertible} and \cref{AlmostHarmonicExistence}) to produce a family of \textit{almost $\Dirac$-harmonic} smooth sections $\set{\ue}_{\epsilon>0}$, i.e.\ that for all $\epsilon>0$ and all $j\geq 1$ 
\begin{equation}
  \Ltwonorm{\ue}=1 \quad \text{and} \quad \Ltwonorm[\big]{\Dirac^j \ue}< \epsilon^j
\end{equation}
hold. The previous estimates on the $\Ltwoblanc$-norm, together with the Schrödinger--Lichnerowicz formula, directly imply that a non-vanishing Rosenberg index is an obstruction to a positive scalar curvature metric on a closed spin manifold. To prove extremality and rigidity statements, the estimates on the $\Ltwoblanc$-norm seem too weak to be geometrically useful. The main part of the proof of \cref{ThmA} is to develop a technique (\cref{SectionAlmostHarmonicSections}) that shows in the extremal geometric situation that the family $\set{\ue}_{\epsilon>0}$ of almost $\Dirac$-harmonic sections is even \textit{almost constant}. This means there exist constants~$C,r>0$ and an element $a\in A^+$ such that 
\begin{equation}
  \Anorm[\big]{a-\scalarproduct[\big]{\ue\brackets{p}}{\ue\brackets{p}}_{p}}<C\epsilon^r
\end{equation}
holds for all $p\in M$ and all $\epsilon\in \brackets{0,1}$. A summary of the technical results: 
\begin{prelem}[\cref{DiracInvertible}, \cref{AlmostHarmonicExistence}, \cref{AlmostHarmonicImpliesAlmostParallel}, \cref{AlmostConstantLem} and \cref{DefAlmost}] \label{LemC}
  Let $\A$ be a graded Real unital $\Cstar$-algebra and $\SpinBdl$ a graded Real $\A$-linear Dirac bundle over a closed connected Riemannian manifold~$M$ with induced Dirac operator~$\Dirac$. If the higher index of~$\Dirac$ does not vanish in~$\KO_0\brackets{\A}$, there exists a family~$\{\ue\}_{\epsilon>0}$ of almost $\Dirac$-harmonic sections. If, moreover, $\Ltwonorm{\nabla \ue}< \epsilon$ holds for all $\epsilon>0$, the family~$\{\ue\}_{\epsilon>0}$ is almost constant.
\end{prelem}
In the proof of \cref{LemC}, we use Moser iteration---similarly as in \cite{Ammann2007}---to deduce from $\Ltwonorm{\nabla \ue}<\epsilon$ estimates on the $\Linftyblanc$-norm (\cref{MoserIteration} and \cref{NablaUeSmallinInfinity}), and the Poincaré inequality to show that~$\{\ue\}_{\epsilon>0}$ is almost constant (\cref{AlmostConstantLem}). \medbreak
In the situation of \cref{ThmA}, \cref{LemC} together with the index formula in \cref{ThmB} and the estimate for~$\DL^2$ yields the existence of a family of almost constant sections~$\{\ue\}_{\epsilon>0}$. This family replaces the $\Dirac$-harmonic section~$u$ in the classical proof of \textcite{Goette2000}. This property of being almost constant is sufficient to prove the extremality and rigidity statement similarly as in the classical proof.
\subsection{A direct proof of a classical result about zero Ricci curvature}
  Let~$M$ be a closed connected spin manifold with non-negative scalar curvature. Denote the spin Dirac operator by~$\Dirac$. The implications in the following diagram hold: 
  \begin{equation} \label{OverviewImplications}
          \begin{tikzcd} 
            \ind \Dirac \neq 0  \arrow[Rightarrow]{d}{} \arrow[Rightarrow]{rr}{}
            && \Ric \equiv 0 \arrow[Rightarrow]{rr}
            && \scal \equiv 0\\
            \alpha \brackets{M} \neq 0 
            \arrow[Rightarrow]{d}
            \arrow[equal]{rrrr}{\text{direct proof of \textcite{Schick2021}}}
            \arrow[Rightarrow, thick]{rru}{(\star)}
            \arrow[Rightarrow, to path={ -- (\tikztostart.east) -| (\tikztotarget)},
            rounded corners=0pt]{rrrru}
            && \textcolor{white}{\overset{X}{a}} \arrow[Rightarrow]{u}{}&& \textcolor{white}{.} \\
            \text{$M$ admits no PSC}\arrow[equal, to path={ -- (\tikztostart.east) -| (\tikztotarget)},
            rounded corners=0pt]{rru}\arrow[equal]{rr}{\text{Bourguignon}} \textcolor{white}{.} \arrow[equal, swap]{rr}{\text{\cite[see e.g.][]{Kazdan1975a}}}&&\textcolor{white}{.} && &&
          \end{tikzcd}
    \end{equation}
    If the classical index of~$\Dirac$ does not vanish, there exists by $\scal\geq 0$ a non-trivial constant section~$u$, and a direct calculation shows that~$M$ is Ricci flat. If the Rosenberg index does not vanish, we obtain a family of almost constant sections by \cref{LemC} and the Schrödinger--Lichnerowicz formula. The existence of such a family is enough to prove that~$M$ is Ricci flat, hence we obtain a direct proof of implication~$(\star)$ in the previous diagram. 
    %
\subsection{Structure}
  In \cref{PreliminariesSection}, we introduce Dirac operators linear over $\Cstar$-algebras and their higher indices, spin structures associated to vector bundles with signature, and the Rosenberg index of a spin manifold. In \cref{SectionAlmostHarmonicSections}, we proof \cref{LemC}. We apply in \cref{SectionZeroRicci} the new developed technique and give a direct proof of the classical result about the Rosenberg index and the Ricci curvature. In \cref{SectionExtensionGoetteSemmelmann}, we proof the extremality and rigidity statement (\cref{ThmA}) and the index theorem (\cref{ThmB}). There we also derive the theorem of \textcite{Goette2000} and treat some special versions of \cref{ThmA} involving the Rosenberg index of the fiber directly (\cref{RigidityThmHigherDergeeRosenberg}). In the end of \cref{SectionExtensionGoetteSemmelmann}, we prove the rigidity statement for fiber bundles (\cref{PreCor}) over even dimensional spheres, quaternionic Grassmannians and the Cayley projective plane.\par
%
\subsection{Acknowledgments}
I thank my advisor Rudolf Zeidler for introducing me to this interesting problem and for sharing his knowledge in many discussions. I also thank Bernd Ammann, Renato Ghini Bettiol and Jonathan Glöckle for valuable conversations, as well as the anonymous referee for useful comments on the manuscript. \par
Funded by the European Union (ERC Starting Grant 101116001 – COMSCAL)\footnote{
  Views and opinions expressed are however those of the author(s) only and do not necessarily reflect those of the European Union or the European Research Council. Neither the European Union nor the granting authority can be held responsible for them.
  %
  %
}
and by the Deutsche Forschungsgemeinschaft (DFG, German Research Foundation) – Project-ID 427320536 – SFB 1442, as well as under Germany’s Excellence Strategy EXC 2044 390685587, Mathematics Münster: Dynamics–Geometry–Structure.
\section{Preliminaries} \label{PreliminariesSection}
In this section, we summarize notations and definitions including $\A$-linear Dirac operators and their higher indices, spin structures for vector bundles with indefinite metric, the Mishchenko line bundle and the Rosenberg index of a closed spin manifold. Furthermore, we give an overview of some important properties round about these concepts. 
\subsection{\texorpdfstring{$\A$}{}-linear Dirac operator and its higher index} \label{AlinearDiracOperatorsSubsection}
Let $\brackets{M,g}$ be an $n$-dimensional complete Riemannian manifold and~$\A$ a unital graded Real $\Cstar$-algebra with norm denoted by $\Anorm{\placeholder}$. At first, we generalize the definition of Dirac bundles and their induced Dirac operators \cite[see][Chapter 2]{Lawson1989} to the corresponding $\A$-linear structures: 
\begin{defi}
  An \textit{$\A$-linear Dirac bundle} over~$M$ is a bundle of finitely generated projective Hilbert $\A$-modules $\SpinBdl \to M$ equipped with an $\A$-valued metric, a compatible $\A$-linear connection and a parallel bundle map $c:TM\to \End_{\A}\brackets{\SpinBdl}$, called the \textit{Clifford multiplication}, such that $c(w)$ is skew-adjoint and satisfies $c(w)^2=-g\brackets{w,w}$ for all $w\in TM$. We call an $\A$-linear Dirac bundle \textit{graded} (\textit{Real}) if it is furnished with a grading (Real structure), which is compatible with the grading (Real structure) of~$\A$, the metric, the connection and the Clifford multiplication~$c$, i.e.~$c\brackets{w}$ is odd (Real) for all $w\in TM$. If a grading and a Real structure exist, they have to be compatible with each other. The induced \textit{($\A$-linear) Dirac operator} is defined as
  \begin{equation}
    \Dirac_{\SpinBdl}\colon \Ccinfty{M,\SpinBdl} \To{\text{ connection }} \Ccinfty{M,T^*M\tens \SpinBdl} \To{\text{ metric, Clifford multiplication }} \Ccinfty{M,\SpinBdl}.
  \end{equation}
\end{defi}
The Dirac operator induced by a (graded Real) $\A$-linear Dirac bundle is, by the same considerations as in the classical case \cite{Lawson1989}, an elliptic formally self-adjoint (odd Real) differential operator of first order. Details about bundles of finitely generated projective Hilbert $\Cstar$-modules are given in \cite{Mishchenko1980}, \cite{Schick2005} and \cite{Ebert2016}, and about $\A$-linear Dirac bundles in \mbox{\cite[\S 7]{Stolz}}.
\begin{const}[Twisted Dirac bundle] \label{ConstructionTwistBundle}
  Let $\SpinBdl \to M$ be a graded Real $\A$-linear Dirac bundle, $\B$ a unital graded Real $\Cstar$-algebra and $E\to M$ a graded Real bundle of finitely generated projective Hilbert $\B$-modules. Taking the fiberwise graded exterior tensor product of Hilbert $\Cstar$-modules with respect to the graded maximal tensor product between the $\Cstar$-algebras $\A$ and~$\B$ defines a graded Real $\A \tensgr \B$-linear Dirac bundle $\SpinBdl \tensgr E \to M$. Here we equip $\SpinBdl \tensgr E \to M$ with the tensor connection, the Clifford multiplication induced from the Clifford multiplication on~$S$, and the Real structure induced from the Real structures on~$\SpinBdl$ and~$E$. If the $\Cstar$-algebra~$\B$ and the Hilbert~$\B$-module~$E$ are not graded, we equip them in the previous construction with the trivial grading. 
\end{const}
\begin{nota} 
  Let $\SpinBdl \to M$ be an $\A$-linear Dirac bundle. Denote by $\scalarproduct{\placeholder}{\placeholder}_p$ the $\A$-valued metric in the fiber of $\SpinBdl$ over a point $p\in M$. Furthermore, define for compactly supported smooth sections $u, v\in \Ct^\infty_{c}\brackets{M,\SpinBdl}$, $p\in M$ and $k,m\in \N_0$ the following scalar products and norms:
  \begin{equation}
    \scalarproduct{u}{v}_p \coloneqq \scalarproduct{u_p}{v_p}_p \in \A \quad 
    \abs{u}_p \coloneqq \scalarproduct{u}{u}_p^{1\slash 2} \in \A^+ \quad 
    \norm{u}_p \coloneqq \Anorm[\big]{\scalarproduct{u}{u}_p^{1\slash 2}} \in \R
  \end{equation}
  \begin{center}
    \begin{tabular}{LCLLLCL}
      \Ltwoscalarproduct{u}{v} &\hspace{-0.3cm}\coloneqq&\hspace{-0.2cm} \int_M \scalarproduct{u}{v}_p \intmathd p \in \A 
      && \Hscalarproduct{k}{u}{v} &\hspace{-0.3cm}\coloneqq&\hspace{-0.2cm} \sum_{j=0}^k \Ltwoscalarproduct[\big]{\nabla^j u}{\nabla^j v} \in \A \\
      \Ltwoabs{u} &\hspace{-0.3cm}\coloneqq&\hspace{-0.2cm} \Ltwoscalarproduct{u}{u}^{1/2} \in \A^+ 
      && \Habs{k}{u} &\hspace{-0.3cm}\coloneqq&\hspace{-0.2cm} \Hscalarproduct{k}{u}{u}^{1/2} \in \A^+ \\
      \Ltwonorm{u} &\hspace{-0.3cm}\coloneqq&\hspace{-0.2cm} \Anorm{\Ltwoscalarproduct{u}{u}^{1/2}} \in \R
      && \Hnorm{k}{u}&\hspace{-0.3cm}\coloneqq&\hspace{-0.2cm} \Anorm{\Hscalarproduct{k}{u}{u}^{1/2}} \in \R \\ 
      \Inftynorm{u} &\hspace{-0.3cm}\coloneqq&\hspace{-0.2cm} \max_{p\in M} \norm{u}_p\in \R
      && \norm{u}_{C^m}&\hspace{-0.3cm}\coloneqq&\hspace{-0.2cm} \sum_{j=0}^m \Inftynorm[\big]{\nabla^j u}\in \R
    \end{tabular}
  \end{center}
  Here $\A^+$ denotes the set of positive elements in the $\Cstar$-algebra~$\A$, and $\Ltwoscalarproduct{\nabla^j u}{\nabla^j v}$ is analogously defined as $\Ltwoscalarproduct{u}{v}$ with the induced $\A$-valued inner product in the fibers of $T^*M^j \tens \SpinBdl$. Be aware of the fact that even if $\Habs{k}{\placeholder}$ is called a norm \cite[see][]{Lance1995}, it does not necessarily fulfill the triangle inequality. To prevent the risk of confusion with the~$\A$- and $\R$-valued norms, we denote any $\R$-valued norm by~$\norm{\placeholder}$ and any $\A$-valued norm by~$\abs{\placeholder}$. The Hilbert $\A$-module $\Hsp{k}{M,\SpinBdl}$ is defined as the closure of the compactly supported smooth sections $\Ct_c^\infty\brackets{M,\SpinBdl}$ in the norm $\Hnorm{k}{\placeholder}$. As classically, we write $\Ltwo{M,\SpinBdl}$ for $\Hsp{0}{M,\SpinBdl}$.
\end{nota}
The next goal is to define the higher index of suitable $\A$-linear differential operators, including certain $\A$-linear Dirac operators and perturbed Dirac operators. The latter plays an important role in the proof of the index theorem (\cref{IndexThm}) in \cref{SectionExtensionGoetteSemmelmann}.\par
Fix a graded Real $\A$-linear Dirac bundle $\SpinBdl\to M$ together with an odd Real elliptic formally self-adjoint $\A$-linear differential operator of first order $B \colon \Ct^\infty_c \brackets{M,\SpinBdl}\to \Ct^\infty_c \brackets{M,\SpinBdl}$. The $\Ltwoblanc$-closure, denoted by $\mathrm{B} \colon \dom\brackets{\mathrm{B}} \to \Ltwo{M,\SpinBdl}$, is self-adjoint and regular \cite[][Theorem 1.14]{Ebert2016} and its (possibly non-discrete) spectrum is Real \cite[][Proposition 1.10]{Ebert2016}. This gives rise to a functional calculus 
\mapdefoneline{\Phi_{\mathrm{B}}}{\Ct\brackets[\big]{\overline{\R}}}{\Lin_{\A}\brackets{\Ltwo{M,\SpinBdl}}}{f}{f\brackets{\mathrm{B}}}
with properties as listed in \cite[Theorem 1.19]{Ebert2016}. Here $\Ct\brackets[\big]{\overline{\R}}$ denotes the graded Real $\Cstar$-algebra of continuous functions $f\colon \R \to \C$ such that the limits to $\pm \infty$ exist, and $\Lin_{\A}\brackets{\Ltwo{M,\SpinBdl}}$ denotes the graded Real $\Cstar$-algebra of adjointable operators on $\Ltwo{M,\SpinBdl}$. \par
Assume furthermore, that the unbounded operator~$\mathrm{B}$ is \textit{Fredholm}, i.e.\ that the adjointable operator $\mathrm{B}\brackets{1+\mathrm{B}^2}^{-1/2}$, defined via the functional calculus of~$\mathrm{B}$, is invertible modulo a compact operator. This data gives directly per definition rise to a class in the picture of the Real K-theory of~$\A$ given in \cite[Definition 3.4]{Ebert2016}. We denote this class by $\ind \mathrm{B}$ and call it the \textit{higher index} of~$\mathrm{B}$.
Following \textcite[\S 9]{Stolz}, we give a concrete construction of the higher index in the Fredholm picture $\KO\brackets{\A}$ of the Real K-theory of~$\A$ \cite[see][Lecture 12]{Roe}. Elements of $\KO_0\brackets{\A}$ are homotopy classes of Fredholm pairs~$\brackets{F,E}$ consisting of a countably generated graded Real Hilbert $\A$-module and an odd Real operator~$F\in \Lin_{\A}\brackets{E}$ with $F-F^*,\;F^2-1 \in \Kom_{\A}\brackets{E}$. Note that the Fredholm picture is equivalent to the KK-theoretic picture~$\KK\brackets{\C,\A}$ used by \textcite[\S 9]{Stolz}. \par
The unbounded operator~$\mathrm{B}$ is Fredholm if and only if it has an \textit{essential spectral gap}, i.e.\ there exists an $\epsilon>0$ such that~$f\brackets{\mathrm{B}}$ is a compact operator for all continuous~$f\colon \R\to \R$ with support in $\brackets{-\epsilon, \epsilon}$ \cite[][Lemma 2.19]{Ebert2016}. Fix such an $\epsilon>0$. For a continuous odd function $g\colon \R \to \sqbrackets{-1,1}$ with $g\brackets{x}=\sgn\brackets{x}$ on $\R \setminus \brackets{-\epsilon,\epsilon}$, the operator $g\brackets{\mathrm{B}}^2-1$ is compact, hence we obtain a Fredholm pair $\brackets{g\brackets{\mathrm{B}},\Ltwo{M,\SpinBdl}}$. This defines a class in the Fredholm picture of the Real K-theory of~$\A$, which is independent of~$\epsilon$ and~$g$, and matches by definition with the higher index of $\mathrm{B}$
\begin{equation}
  \ind \mathrm{B} = \sqbrackets[\Big]{\brackets[\big]{g\brackets{\mathrm{B}},\Ltwo{M,\SpinBdl}}} \in \KO_0\brackets{\A}\cong \KK_0\brackets{\C,\A}.
\end{equation} 
The higher index is a cut-and-paste invariance as explained in \cite[Section2.3]{Cecchini2020} and satisfies sum and product formula as mentioned in \cite[Appendix B]{Zeidler2022}. Furthermore, the higher index of a Dirac operator induced by a graded Real $\A$-linear Dirac bundle is independent of the involved metrics by bordism invariance. The latter is a consequence of the partitioned manifold index theorem \cite[see for instance][Appendix A]{Zeidler2022}. \par
For the Fredholmness of~$\mathrm{B}$, it is sufficient that $\mathrm{B}^2$ is \textit{uniformly positive at infinity} \cite[][Theorem 2.41]{Ebert2016}, i.e.\ there exists a compact subset~$K\subset M$ and a constant~$c>0$ such that
\begin{equation} \label{DefUniformlypositiveatInifnityEq}
  \Ltwoscalarproduct[\big]{\mathrm{B}^2 u}{u}\geq c\Ltwoscalarproduct{u}{u}, \quad \forall u\in \text{dom}\brackets[\big]{\mathrm{B}^2} \text{ with } \supp\brackets{u}\cap K=\emptyset.
\end{equation}
If $\mathrm{B}^2$ is \textit{uniformly positive}, i.e.\ \cref{DefUniformlypositiveatInifnityEq} holds for $K=\emptyset$, the higher index vanishes: 
\begin{lem} \label{DiracInvertible}
  Let $\SpinBdl\to M$ be a graded Real $\A$-linear Dirac bundle and~$\mathrm{B}\colon \dom \brackets{\mathrm{B}}\to \Ltwo{M,\SpinBdl}$ the $\Ltwoblanc$-closure of an odd Real elliptic formally self-adjoint differential operator of first order. Then the following holds:
  \begin{myenumi}
    \item 
      If $\mathrm{B}^2$ is uniformly positive, $\mathrm{B}$ is invertible. 
    \item
      If $\mathrm{B}$ is invertible, $\mathrm{B}$ is Fredholm and its higher index vanishes in~$\KO_0\brackets{\A}$.
  \end{myenumi}
\end{lem}
%
%
\begin{proof}
  For the first statement see \cite[Proposition 1.20 and 1.21]{Ebert2016}. Suppose~$\mathrm{B}$ is invertible. Then zero is not in the spectrum of~$\mathrm{B}$, and there exists an $\epsilon>0$ such that $\brackets{-\epsilon, \epsilon}$ does not intersect with the spectrum of~$\mathrm{B}$. Since $f\brackets{\mathrm{B}}$ just depends on the values the function~$f$ takes on the spectrum of~$\mathrm{B}$, we obtain $f\brackets{\mathrm{B}}=0$ for any function~$f$ supported in $\brackets{-\epsilon, \epsilon}$, hence~$\mathrm{B}$ has an essential spectral gap. This shows that~$\mathrm{B}$ is Fredholm. Furthermore, $g\brackets{\mathrm{B}}^2-1$ is equal to zero for a function~$g$ as in the construction of the higher index, hence $\brackets{g\brackets{\mathrm{B}},\Ltwo{M,\SpinBdl}}$ vanishes in $\KO_0\brackets{\A}$ by \cite[Lemma 12.6]{Roe}.
\end{proof}
So far, the underlying manifold was not necessarily closed. From now on assume that the manifold~$M$ is closed. Then the induced Dirac operator~$\Dirac$ of an $\A$-linear Dirac bundle $\SpinBdl\to M$ with initial domain $\Cinfty{M,\SpinBdl}$ extends for all $k\in \N_0$ to an unbounded operator $\Dirac:\Hsp{k+1}{M,\SpinBdl} \to \Hsp{k}{M,\SpinBdl}$ \cite[][Lemma 3.2]{Mishchenko1980}. Furthermore, the $\Ltwoblanc$-extended Dirac operator is obviously uniformly positive at infinity, hence Fredholm. 
\par
The classical elliptic regularity principle carries over to $\A$-linear Dirac operators, which means precisely that the classical Sobolev embedding theorem \cite[Chapter III Theorem 2.5 and Theorem 2.15]{Lawson1989} and the classical fundamental elliptic estimate \cite[Chapter III Theorem 5.2]{Lawson1989} generalizes as follows: 
\begin{lem} \label{ClassicalEstimatesExtension}
  Let $k,m, l\in \N_0$ with $k>m+\frac{n}{2}$ and $\SpinBdl\to M$ an $\A$-linear Dirac bundle with induced Dirac operator~$\Dirac$. There exists a constant $C$ such that for all $u\in \Hsp{l+1}{M,\SpinBdl}$ and all $v\in \Hsp{k}{M,\SpinBdl}$ the following holds:
  \begin{myenumi}
    \begin{samepage}
      \item \(\norm{v}_{C^m}\leq C \Hnorm{k}{v}\) \hfill (Sobolev inequality)
      \item \(\norm{u}_{H^{l+1}}\leq C \brackets[\big]{\norm{\Dirac u}_{H^{l}}+\norm{u}_{H^{l}}}\) \hfill (fundamental elliptic estimate)
    \end{samepage}
  \end{myenumi}
\end{lem}
\begin{proof}[Sketch of proof]
  Since the manifold~$M$ is closed, it is as in the classical case enough to show the estimates for the trivial bundle $T^n\times P$ over the flat torus and a first-order elliptic differential operator. Here~$P$ denotes a finitely generated projective Hilbert $\A$-module. The general case follows by considering an appropriate partition of unity and picking a good system of trivializations of the bundle~$\SpinBdl$ as explained in \cite[Chapter III \S2]{Lawson1989}. On the torus, we can use Fourier transform methods, which generalize from $\C^m$-valued functions to functions with values in the Hilbert $\A$-module~$P$ \cite[see][\S3]{Mishchenko1980}. This implies the Sobolev inequality by the same considerations as in \cite[Chapter III \S2]{Lawson1989}. The fundamental elliptic estimate follows by induction from the Gårding's inequality, whose classical proof in \cite[Section 10.4.4]{Higson2000} carries over without change.
\end{proof}
\subsection{Spin structure of vector bundles with indefinite metric} \label{SpinStructureSubsection}
In this section, we define spin structures for vector bundles equipped with a non-degenerate not necessarily positive definite metric. For this purpose, we mix the definition of spin structures for vector bundles with a positive definite metric \cite[][Chapter II]{Lawson1989} with the notion of a spin structure for semi-Riemannian manifolds \cite[][]{Baum1981, Chen2012, Ammann2023}.\medbreak
For~$m,n\in \N_0$ denote by~$\ComplexCl_{m,n}$ the graded Real $\Cstar$-algebra generated by $\R^m\oplus \R^n$ such that $\brackets{v,w}^*=\brackets{-v,w}$ and
\begin{equation}
  \brackets{v,w}\brackets{v',w'}+\brackets{v',w'}\brackets{v,w}=-2\scalarproduct{v}{v'}_{\text{eu}}+2\scalarproduct{w}{w'}_{\text{eu}}
\end{equation}
hold for all $\brackets{v,w},\brackets{v',w'}\in \R^m\oplus \R^n$, and all elements of $\R^m\oplus \R^n$ are odd \cite[see][Section 1.1]{Ebert2016}. The pin group~$\Pin_{m,n}$ is the subgroup of~$\ComplexCl_{m,n}$ generated by $\brackets{v,w}\in \R^n\oplus \R^m\subset \ComplexCl_{m,n}$ with $\norm{v}_{eu}-\norm{w}_{eu}=\pm 1$. Furthermore, we define the subgroups $\Spin_{m,n}$ and $\Spin^+_{m,n}$ of $\Pin_{m,n}$ as the preimage of $\SO\brackets{m,n}$ and $\SO^+\brackets{m,n}$, respectively, along the two-folded covering
\begin{equation}
  \lambda\colon \Pin_{m,n} \To{} \Orth\brackets{m,n}, \quad \lambda\brackets{q}\brackets{x} \coloneqq \brackets{-1}^{\deg\brackets{q}}q x q^{-1}.
\end{equation}
Here $\SO^+\brackets{m,n}$ denotes the identity component of $\SO\brackets{m,n}$. The spin group $\Spin_{m,n}$ matches with the even part of $\Pin_{m,n}$, and $\Spin^+_{m,n}$ is equal to the identity component of~$\Spin_{m,n}$. \medbreak
Let~$M$ be a connected manifold and~$V$ an orientable vector bundle of rank~$m+n\geq 3$ equipped with a metric~$h$ of signature~$\brackets{m,n}$. The existence of such a metric is equivalent to the fact that the frame bundle~$\mathcal{F}\brackets{V}$ reduces to a $\Orth\brackets{m,n}$-principle bundle~$\Orth\brackets{V,h}$ \cite[Satz 0.47]{Baum1981}. We fix a splitting of~$V$ into two orthogonal sub-bundles~$V_1$ and~$V_2$ such that the induced metrics on~$V_1$ and~$V_2$ are positive and negative definite, respectively. The Stiefel--Whitney classes of~$V_1$ and~$V_2$ do not depend on that choice. If~$V_1$ and~$V_2$ are both orientable, the principle bundle~$\Orth\brackets{V,h}$ has four connected components and it reduces to a connected $\SO^+\brackets{m,n}$-principle bundle~$\SO\brackets{V,h}$. Otherwise, it has 2-connected components and reduces to a connected $\SO\brackets{m,n}$-principle bundle~$\SO\brackets{V,h}$ \cite[Satz 0.51]{Baum1981}. From now on, we distinguish between these two cases and write~\enquote{$+$} in parentheses.\par
A \textit{spin structure} of~$\brackets{V,h}$ is a reduction of~$\SO^{\brackets{+}}\brackets{V,h}$ to a $\Spin^{\brackets{+}}_{m,n}$-principle bundle~$\Spin\brackets{V,h}$ along the restriction of~$\lambda$ onto $\Spin^{\brackets{+}}_{m,n}$. An overview of all reductions and lifts of the frame bundle of~$V$ is given in \cref{IndefinitSpinStructure}.
\begin{figure}
\begin{center}
    \begin{tikzcd}[row sep=scriptsize, column sep=scriptsize]
        \Spin^{\brackets{+}}_{m,n} \arrow[->>]{rrr}{\lambda\vert_{\Spin^{\brackets{+}}_{m,n}}} &&& 
        \SO^{\brackets{+}}\brackets{m,n} \arrow[rr, hook] && 
        \Orth\brackets{m,n} \arrow[rr, hook] && 
        \GL\brackets{m+n}
        \\
        \Spin^{\brackets{+}}\brackets{V,h} \arrow[->>]{rrr}{\text{double cover}} \arrow[loop above] \arrow{rrrdd} &&& 
        \SO^{\brackets{+}}\brackets{V,h} \arrow[rr, hook] \arrow[loop above] \arrow{dd}&& 
        \Orth\brackets{V,h} \arrow[rr, hook] \arrow[loop above] \arrow{lldd}&& 
        \mathcal{F}\brackets{V} \arrow[loop above] \arrow{lllldd}
        \\ \\
        \text{spin structure} \arrow[bend left=25,swap, dashed]{uu}&&& 
        \quad M \quad &&&&
  \end{tikzcd} 
\end{center}
\caption{Spin structure of a rank~$m+n$ vector bundle~$V$ over a connected manifold~$M$ equipped with a metric~$h$ of signature~$(m,n)$.}
\label{IndefinitSpinStructure}
\end{figure}
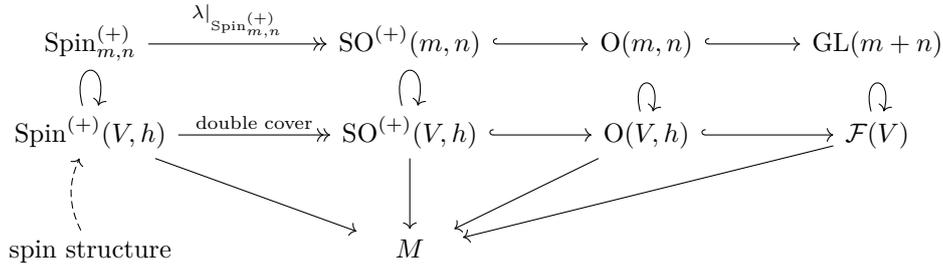
There exists a spin structure for~$\brackets{V,h}$ if and only if $w_2\brackets{V_1}+w_2\brackets{V_2}=0$ holds \cite[Satz 2.2]{Baum1981}. Note that if both sub-bundles~$V_1$ and~$V_2$ are orientable, this condition is equivalent to
\begin{equation}
  w_2\brackets{V}=w_2\brackets{V_1}+w_2\brackets{V_2}+w_1\brackets{V_1}w_1\brackets{V_2}=0,
\end{equation}
hence it is equivalent to the existence of a \textit{topological spin structure} on~$V$. This is a lift of $\mathcal{F}^+\brackets{V}$ along the two-fold covering $\widetilde{\GL^+}\brackets{m+n} \to \GL^+\brackets{m+n}$, where $\mathcal{F}^+\brackets{V}$ denotes a reduction
of $\mathcal{F}\brackets{V}$ along the inclusion $\GL^+\brackets{m+n}\xhookrightarrow{} \GL\brackets{m+n}$. 
\par
In the remaining section, we will construct the $\ComplexCl_{m,n}$-linear spinor bundle from a given spin structure on $\brackets{V,h}$. Let $K_{m,n}^{(\pm)}\subset \Spin^{(\pm)}_{m,n}$ denote the maximal compact subgroups covering $S\brackets{O_m\times O_n}\subset \SO_{m,n}$ and $\SO_m\times \SO_n\subset \SO^+_{m,n}$, respectively. Since $K_{m,n}^{(\pm)}\subset \Spin_{m,n}^{(\pm)}$ is a deformation retract, any $\Spin^{(\pm)}_{m,n}$-principle bundle reduces to a $K^{(\pm)}_{m,n}$-principle bundle. Therefore, a fixed spin structure on~$\brackets{V,h}$ gives rise to a $K^{(\pm)}_{m,n}$-principle bundle $P_{K^{(\pm)}_{m,n}}\brackets{V,h}$, and we obtain a vector bundle
\begin{equation}
  \SpinBdl\brackets{V,h}\coloneqq P_{K^{(\pm)}_{m,n}}\brackets{V,h} \times_{K^{(\pm)}_{m,n}} \ComplexCl_{m,n},
\end{equation}
where $K^{(\pm)}_{m,n}\subset \ComplexCl_{m,n}$ acts as left multiplication on $\ComplexCl_{m,n}$. The vector bundle~$\SpinBdl\brackets{V,h}$ inherits a Real structure, a grading and a right $\ComplexCl_{m,n}$-action. The fiberwise canonical $\ComplexCl_{m,n}$-valued metric on $\ComplexCl_{m,n}$ defines a global positive definite $\ComplexCl_{m,n}$-valued metric on $\SpinBdl\brackets{V,h}$. Here it is important that we used the reduction of the $\Spin^{(\pm)}_{m,n}$-principle bundle $\Spin^{\brackets{+}}\brackets{V,h}$ to the subgroup $K^{(\pm)}_{m,n}$ in the construction of $\SpinBdl\brackets{V,h}$.
The map $\brackets{\R^m\oplus \R^n} \times \ComplexCl_{m,n} \to \ComplexCl_{m,n}$ given by left multiplication induces a well-defined bundle map
\begin{equation}
  \clm2 \colon V\cong \SO^{\brackets{+}}_{m,n}\brackets{m,n}\times_{\SO^{\brackets{+}}_{m,n}}\brackets{\R^m\oplus \R^n} \To{} \End_{\ComplexCl_{m,n}}\brackets{\SpinBdl\brackets{V,h}}, 
\end{equation} 
which is odd, Real, satisfies $\clm2\brackets{v}^2=-h\brackets{v,v}$ for all $v\in V$, and $\clm2\brackets{v}$ is skew-adjoint for $v\in V_1$ and self-adjoint for $v\in V_2$. If the vector bundle $V$ is equipped with a metric connection, the vector bundle~$\SpinBdl\brackets{V,h}$ inherits a canonical connection making all structures parallel. We call~$\SpinBdl\brackets{V,h}$ the \textit{$\ComplexCl_{m,n}$-linear spinor bundle induced by~$\Spin \brackets{V,h}$}. \par
For an orientable Riemannian manifold~$\brackets{M,g}$ of dimension~$m$ with non-vanishing second Stiefel--Whitney class the previous construction with $\brackets{V,h}=\brackets{TM,g}$ matches with the classical spin structure in \cite[Chapter II \S1 and \S2]{Lawson1989}. If the tangent bundle is equipped with the Levi-Civita connection, the induced spinor bundle $\SpinBdl\coloneqq \SpinBdl\brackets{TM,g}$ is the $\ComplexCl_m$-linear Dirac bundle in \cite[Chapter II \S8]{Lawson1989}.
\subsection{The Mishchenko bundle and the Rosenberg index} \label{MishchenkoBundleSubsection}
Let~$M$ be a closed connected manifold of dimension~$m$ with fundamental group $\pi$ and universal cover $\widetilde{M}$. Denote by~$\Cstar \pi$ the maximal group $\Cstar$-algebra of~$\pi$ equipped with the Real structure induced by the natural Real structure on~$\C$. The \textit{Mishchenko (line) bundle} of~$M$ \cite{Mishchenko1980} is defined by
\begin{equation}
  \mathcal{L}\brackets{M}\coloneqq \mathcal{L} \coloneqq \widetilde{M} \times_{\pi} \Cstar \pi,
\end{equation}
where~$\pi$ acts on~$\Cstar \pi$ by left multiplication. The bundle~$\mathcal{L}$ comes with a canonical flat connection and a $\Cstar \pi$-valued product given by $\scalarproduct{a}{b}\coloneqq a^*b$, hence~$\mathcal{L}$ is a flat bundle of finitely generated projective Hilbert $\Cstar \pi$-modules together with a Real structure.\par
Suppose the manifold~$M$ is equipped with a spin structure. Twisting the $\ComplexCl_m$-linear spinor bundle~$\SpinBdl$ of~$M$ by the Mishchenko bundle $\mathcal{L}$ as in \cref{ConstructionTwistBundle} gives a graded Real $\ComplexCl_{m}\tensgr \Cstar \pi$-linear Dirac bundle $\SpinBdl_{\mathcal{L}} \coloneqq \SpinBdl \tensgr \mathcal{L}$. The higher index of the induced $\ComplexCl_m\tensgr\Cstar \pi$-linear Dirac operator~$\Dirac_{\mathcal{L}}$, denoted by 
\begin{equation} \label{DefRosenbergIndexEq}
  \alpha\brackets{M} \coloneqq \ind\brackets{\Dirac_{\mathcal{L}}} \in \KO_0\brackets{\ComplexCl_m\tensgr\Cstar \pi} \cong \KO_m\brackets{\Cstar \pi},
\end{equation}
is called the \textit{Rosenberg index of~$M$} \cite{Rosenberg1983}. There is a well-known connection between the Rosenberg index and the scalar curvature of~$M$:
\begin{prop}
  If the scalar curvature of the manifold~$M$ is positive, then the Rosenberg index of~$M$ vanishes.  
\end{prop}
\begin{proof}
  Assume $\scal_M>0$. Since the Mishchenko bundle~$\mathcal{L}$ is flat, the Schrödinger--Lichnerowicz formula yields $\Dirac_{\mathcal{L}}^2=\nabla^*\nabla+\frac{\scal_M}{4}$, and we obtain $\Ltwoscalarproduct{\Dirac_{\mathcal{L}}^2u}{u}\geq \min_{p\in M} \scal_M\brackets{p}\Ltwoscalarproduct{u}{u}$ for all $u\in \Hsp{2}{M,\SpinBdl_{\mathcal{L}}}$. Note that $\min_{p\in M} \scal_M\brackets{p}$ is strictly positive by assumption on the scalar curvature. By \cref{DiracInvertible} part~(1), the Dirac operator~$\Dirac_{\mathcal{L}}$ is invertible and the Rosenberg index of~$M$ vanishes by part~(2) of the same lemma.    
\end{proof}
There is a meta-conjecture by \textcite[Conjecture 1.5]{Schick2014}, that any obstruction to positive scalar curvature on a closed spin manifold (of dimension at least five) based on index theory of Dirac operators can be read off the Rosenberg index. The following proposition is a well-known result in this direction and justifies why the Mishchenko bundle is the right bundle for the twist with the spinor bundle. In particular, it states  that any obstruction to a positive scalar curvature metric on~$M$, obtained from a higher index of the spin Dirac operator twisted with any Real flat bundle of finitely generated projective Hilbert $\A$-modules, can be read off the Rosenberg index. 
\begin{prop} \label{WhyMishchenkoProp}
  Let $\SpinBdl\to M$ be a graded Real $\A$-linear Dirac bundle, $\B$ a unital Real $\Cstar$-algebra and $E\to M$ a Real flat bundle of finitely generated projective $\B$-modules. Denote by~$\DE$ and $\Dirac_{\mathcal{L}}$ the induced Dirac operators of~$\SpinBdl$ twisted with the bundle~$E$, respectively the Mishchenko bundle~$\mathcal{L}$. Then the following holds:
  \begin{equation}
    \ind \DE \neq 0 \in \KO\brackets{\A\tensgr\B}
    \quad \Longrightarrow \quad
    \ind\Dirac_{\mathcal{L}} \neq 0\in \KO\brackets{\A\tensgr\Cstar \pi} 
  \end{equation}
\end{prop}
\begin{proof}
  We orientate towards the proof of a similar, but slightly weaker, result by \textcite[Proposition 3.1 and Theorem 3.2]{Schick2021a}. Let~$P$ be the typical fiber of the flat bundle~$E$. Then there exists a representation $\rho \colon \pi \to U_{\B}\brackets{P}$ on the unitary operators on the Real finitely generated projective Hilbert $\B$-module~$P$ such that $E\cong \widetilde{M}\times_{\rho} P$. By the universal property of the maximal group $\Cstar$-algebra, this representation extends to a $\Cstar$-algebra homomorphism $\overline{\rho}\colon \Cstar \pi \to \Lin_{\B}\brackets{P}$ that respects the grading and the Real structure. The representation
  \begin{equation}
    \psi\coloneqq \id_{\A}\tensgr \overline{\rho}:\A\tensgr \Cstar \pi \to \A \tensgr \Lin_{\B}\brackets{P}\cong \Lin_{\A\tensgr\B}\brackets{\A\tensgr P}
  \end{equation}
  induces as in the proof of \cite[Lemma 12.17]{Roe} a homomorphism on Real K-theory, which takes in the Fredholm picture the form
  \mapdefoneline
    {\brackets{\psi}_0}
    {\KO_0\brackets{\A \tensgr\Cstar \pi}}{\KO_0\brackets{\A \tensgr\B}}
    {\sqbrackets{\brackets{F,Q}}}
    {\sqbrackets{\brackets{F\tens \id_{\A\tensgr P},Q\tens_{\psi}\brackets{\A\tensgr P}}}.}
  Here $\tens_{\psi}$ denotes the interior tensor product with respect to the representation~$\psi$. Furthermore, the bundle~$E$ is related to the Mishchenko bundle via $E \cong \widetilde{M}\times_{\rho} P \cong \widetilde{M}\times_{\rho} \brackets{\Cstar \pi\tens_{\overline{\rho}}P }\cong \mathcal{L}\tens_{\overline{\rho}}P$, which implies 
  \begin{equation}
    \Ltwo{M,\SpinBdl\tensgr \mathcal{L}}\tens_{\psi}\brackets{\A\tensgr P}
    \cong \Ltwo{M,\brackets{\SpinBdl\tensgr \mathcal{L}}\tens_{\psi}\brackets{\A\tensgr P}}
    \cong \Ltwo{M,\SpinBdl \tensgr E}.
  \end{equation}
  Note that the tensor product of a bundle of Hilbert $\Cstar$-modules with a Hilbert $\Cstar$-module is always taken fiberwise. Together with $\Dirac_{\mathcal{L}} \tens \id_{A \tensgr P}=\DE$, we obtain
  \begin{equation}
    \brackets{\psi}_{0}\colon \ind \Dirac_{\mathcal{L}}=\sqbrackets{\brackets{\Dirac_{\mathcal{L}},\Ltwo{M,S\tensgr\mathcal{L}}}} \Mapsto 
    \ind \DE = \sqbrackets{\brackets{\DE,\Ltwo{M,S \tensgr E}}},
  \end{equation}
  and the proposition is proved. 
\end{proof}
%
\section{Almost harmonic, almost parallel and almost constant sections} \label{SectionAlmostHarmonicSections}
Let $\SpinBdl$ be a classical graded Dirac bundle over a closed connected Riemannian manifold~$M$. If the index of the induced Dirac operator $\Dirac$ is non-zero, there exists a non-trivial smooth section~$u$ in the kernel of $\Dirac$. Furthermore, if~$u$ is additionally parallel, i.e.\ $\nabla u=0$, the map $f\colon M \to \R$, $p\mapsto \norm{u}_p$, is constant. In this section, we generalize this result to an $\A$-linear Dirac operator induced by a graded Real $\A$-linear Dirac bundle over a closed connected Riemannian manifold. It turns out that if the higher index of the Dirac operator does not vanish in $\KO\brackets{\A}$, there exists a family of almost $\Dirac$-harmonic sections (see \cref{DefAlmost}). We will furthermore study under which conditions these sections are almost constant in the sense of \cref{DefAlmost}.\par
At first, we make the notion of almost $\Dirac$-harmonic, almost parallel and almost constant precise. Second, we construct almost $\Dirac$-harmonic sections under the hypothesis that the $\A$-linear Dirac operator is not invertible (\cref{AlmostHarmonicExistence}), which is the case, if its higher index does not vanish. Last, we study the relation between almost $\Dirac$-harmonic, almost parallel and almost constant sections in general. \medbreak
We fix a graded Real $\A$-linear Dirac bundle~$\SpinBdl$ over a closed connected Riemannian manifold~$M$ of dimension~$n$ with induced Dirac operator~$\Dirac$. 
\begin{defi} \label{DefAlmost}
  A~family $\set{\ue}_{\epsilon>0}$ of smooth sections of~$\SpinBdl$, which is $\Ltwoblanc$-normalized, i.e.\ $\Ltwonorm{\ue}=1$ for all $\epsilon >0$, is called 
  \begin{enumerate}
    \item 
      \textit{almost $\Dirac$-harmonic} if $\Ltwonorm[\big]{\Dirac^j\ue}<\epsilon^j$ for all $\epsilon > 0$ and all $j\geq 1$.
    \item 
      \textit{almost parallel} if it satisfies $\Ltwonorm{\nabla \ue}<\epsilon$ for all $\epsilon>0$ and all covariant derivatives are uniformly bounded in $\epsilon\in \brackets{0,1}$, i.e.\ for all $m\in \N_0$ there exists a constant~$C_m$ such that $\Inftynorm{\nabla^m\ue}< C_m$ for all $\epsilon \in \brackets{0,1}$.
    \item 
      \textit{almost constant} if there exist constants~$C,r>0$ such that
      \begin{equation}
        \Anorm[\big]{\overline{u}_\epsilon-\abs{\ue}_p^2}< C  \epsilon^r, \quad \text{with} \quad
        \overline{u}_\epsilon \coloneqq \frac{1}{\vol M}\int_M \abs{\ue}^2_q \intmathd q,
      \end{equation}
      for all $\epsilon \in (0,1)$ and all $p\in M$.
  \end{enumerate}
\end{defi}
In the next lemma, we give a connection between the existence of a family of almost $\Dirac$-harmonic sections and the invertibility of the Dirac operator~$\Dirac$. Combined with \cref{DiracInvertible} (2), we obtain: If the index of~$\Dirac$ does not vanish in~$\KO_0\brackets{\A}$, there exists a family of almost $\Dirac$-harmonic sections. 
\begin{lem} \label{AlmostHarmonicExistence}
  If the Dirac operator~$\Dirac$ is not invertible, there exists a family of almost $\Dirac$-harmonic sections.
\end{lem}
\begin{proof}
  We fix for all $\epsilon>0$ a smooth function $\feblanc\colon \R \to \R$ with $\fe{\R}\subset \sqbrackets{0,2}$, $\supp  \brackets{\feblanc}\subset \brackets{-\frac{\epsilon}{2},\frac{\epsilon}{2}}$ and $\fe{0}=2$. By the functional calculus, we obtain adjointable operators $\fe{\Dirac}\colon \Ltwo{M,\SpinBdl}\to \Ltwo{M,\SpinBdl}$ satisfying 
  \begin{equation} \label{opnormfepsilon}
  \opnorm{\fe{\Dirac}}=\sup_{t\in \sigma\brackets{\Dirac}} \abs{\fe{t}}=2.
  \end{equation}
  Here we used in the second step that zero is in the spectrum of $\Dirac$. We conclude from \cref{opnormfepsilon} that there exists $\widetilde{v}_\epsilon \in \Ltwo{M,\SpinBdl}$ with $\Ltwonorm{\widetilde{v}_\epsilon}=1$ and $\Ltwonorm{\fe{\Dirac}\widetilde{v}_\epsilon}\geq 1$. Define a family of sections via
  \begin{equation} \label{Defue}
    u_{\epsilon} \coloneqq \fe{\Dirac}\ve\in \Ltwo{M,\SpinBdl} \quad \text{with}  \quad \ve \coloneqq \frac{\widetilde{v}_\epsilon}{\Ltwonorm{\fe{\Dirac}\widetilde{v}_\epsilon}}.
  \end{equation}
  All~$\ue$ satisfy $\Ltwonorm{\ue}=1$ by definition, hence it remains to show that all~$\ue$ are smooth, and the family is indeed almost $\Dirac$-harmonic. We fix an arbitrary $j\in \N$ and $\epsilon >0$. Using the properties of the functional calculus as in \cite[Theorem 1.19]{Ebert2016} inductively gives 
  that $\ue$ is in the domain of $\Dirac^j$ and $\Dirac^j \ue$ equals $g_\epsilon^j\brackets{\Dirac} \ve$ for $g_\epsilon^j\brackets{t}=t^j \fe{t}, t\in \R$. This gives 
  \begin{equation}
    \Ltwonorm[\big]{\Dirac^j u_{\epsilon}}
    =\Ltwonorm[\big]{g_\epsilon^j\brackets{\Dirac} \ve}
    \leq \opnorm[\big]{g_\epsilon^j\brackets{\Dirac}} \Underbrace{\Ltwonorm{\ve}}{\leq 1}
    \leq \sup_{t\in \sigma \brackets{\Dirac}}\abs[\big]{t^j\fe{t}} 
    < \epsilon^j.
  \end{equation}
  The domain of~$\Dirac^j$ equals $\Hsp{j}{M,\SpinBdl}$, so that~$\ue$ is in $\Hsp{j}{M,\SpinBdl}$ by the previous considerations. It follows, by the Sobolev inequality in \cref{ClassicalEstimatesExtension}~(i), that~$\ue$ is smooth, and the lemma is proved. 
\end{proof}
We will see in \cref{AlmostConstantLem}, that any family of almost parallel sections is almost constant. This turns the existence of a family of almost parallel sections into a useful tool to deduce geometric consequence, as the extremality and rigidity statements in \cref{SectionExtensionGoetteSemmelmann}. The following lemma connects the existence of almost $\Dirac$-harmonic sections with the existence of almost parallel sections: 
\begin{lem} \label{AlmostHarmonicImpliesAlmostParallel}
  Any family $\set{\ue}_{\epsilon>0}$ of almost $\Dirac$-harmonic sections with $\Ltwonorm{\nabla\ue}<\epsilon$ for all $\epsilon>0$ is almost parallel. 
\end{lem}
\begin{proof}
  Let $\set{\ue}_{\epsilon>0}$ be a family of almost $\Dirac$-harmonic sections satisfying $\Ltwonorm{\nabla\ue}<\epsilon$ for all $\epsilon>0$. We fix an $m\geq 0$ and a $k>m+\frac{n}{2}$. There exist, by the Sobolev inequality in \cref{ClassicalEstimatesExtension}~(i) and an iterated version of the fundamental elliptic estimate in \cref{ClassicalEstimatesExtension}~(ii), positive constants $C_1$ and $C_m$ such that
  \begin{equation}
    \Inftynorm{\nabla^m \ue}
    \leq \norm{\ue}_{C^m}
    \leq C_1\Hnorm{k}{\ue}
    \leq C_m\sum_{j=0}^k \Ltwonorm[\big]{\Dirac^j \ue}
    < C_m (1+k) 
  \end{equation}
  holds for all $\epsilon \in \brackets{0,1}$. Here we used in the first step the definition of the $C^m$-norm and in the last step that $\set{\ue}_{\epsilon >0}$ is almost $\Dirac$-harmonic. This shows that for all~$m\in \N_0$ the $m$-th covariant derivative of~$\ue$ is uniformly bounded in $\epsilon \in \brackets{0,1}$, hence $\set{\ue}_{\epsilon>0}$ is almost parallel. 
\end{proof}
The remaining section is dedicated to the statement that any family of almost parallel sections is almost constant (\cref{AlmostConstantLem}). In the first step we use Moser iteration (\cref{MoserIteration}) to deduce from the $\Ltwoblanc$-estimates on the covariant derivations on~$\ue$, estimates on the $\Linftyblanc$-norm. It is crucial for this step, that all covariant derivatives of~$\ue$ are uniformly bounded in $\epsilon \in \brackets{0,1}$. In the second step, we use the Poincaré inequality to show that any family of almost parallel sections is almost constant.\medbreak
The version of Moser iteration---as stated in the following lemma---is a combination of \cite[Theorem 9.2.7]{Petersen2016} and the corrected version of \cite[Lemma 3.1]{Petersen1999} in the erratum \cite[Lemma 0.1]{Petersen2003}. We use the non-negative Laplace operator defined via $\Laplace f \coloneqq -\tr\brackets{\Hess f}$. 
\begin{lem}[Moser iteration] \label{MoserIteration}
  Let $f\colon M\to \R^+\coloneqq \R_{\geq 0}$ be a continuous function that is smooth on $\set{f>0}$ and satisfies $\Laplace f \leq \alpha$ on $\set{f>0}$ for a positive constant $\alpha$. Then there exist constants $C>0$ and $r\in (e^{-\brackets{n-2}\ln\brackets{2}/2},1]$ such that 
    \begin{equation}
      \Inftynorm{f}\leq C\Ltwonorm{f}^r
    \end{equation}
  holds. The constant~$C$ just depends on the constant~$\alpha$ and the dimension, the diameter and the Ricci curvature of the manifold~$M$. 
\end{lem}
\begin{lem} \label{NablaUeSmallinInfinity}
  Let $\set{\ue}_{\epsilon >0}$ be a family of almost parallel sections. Then there exist constants~$C,r>0$ such that $\Inftynorm{\nabla \ue}< C \epsilon^r$ holds for all $\epsilon \in (0,1)$. 
\end{lem}
\begin{proof}
	Let $\set{\ue}_{\epsilon>0}$ be a family of almost parallel sections. Note that the norm of a self-adjoint element $a\in\A$ is given by
  \begin{equation} \label{EqNormState1}
    \Anorm{a}=\max_{\rho \text{ state of }\A} \abs{\rho\brackets{a}},
  \end{equation}
  where a state of the $\Cstar$-algebra~$\A$ is a positive normalized functional $\rho\colon \A\to \C$ \cite[Section 3.3]{Murphy1990}. We define for any $\epsilon \in \brackets{0,1}$ and any state~$\rho$ of the $\Cstar$-algebra~$\A$ a function
  \mapdefoneline{f_{\epsilon, \rho}}{M}{\R^+}{p}{\rho\brackets[\big]{\abs{\nabla \ue}^2_p}^{1/2}.}
  Since all $\ue$ are smooth, all functions $f_{\epsilon, \rho}$ are continuous and smooth on $\set{f_{\epsilon, \rho}>0}$. We will show, that there exists a constant $\alpha>0$, independent of $\epsilon$ and the state $\rho$, such that 
  \begin{equation} \label{ConditionMoser}
    \Laplace f_{\epsilon, \rho}\brackets{p}\leq \alpha 
  \end{equation}
  holds for all $p\in M$ with $f_{\epsilon, \rho}\brackets{p}\neq 0$, all $\epsilon \in \brackets{0,1}$ and all states~$\rho$. Moser iteration (\cref{MoserIteration}) would be applicable and would give for a suitable constant $C>0$, independent of $\epsilon$ and $\rho$, and constants $r_{\epsilon, \rho}\in (e^{-\brackets{n-2}\ln\brackets{2}/2},1]$ 
  \begineqarray \label{MoserIterationForf}
      \rho \brackets[\big]{\abs{\nabla \ue}^2_p}\leq \Inftynorm{f_{\epsilon, \rho}}^2\leq C^2\Ltwonorm{f_{\epsilon, \rho}}^{2r_{\epsilon, \rho}}
      &=&C^2 \brackets{\int_M \rho\brackets[\big]{\abs{\nabla \ue}^2_p} \intmathd p}^{r_{\epsilon, \rho}} \\
      &\leq& C^2 \Ltwonorm{\nabla \ue}^{2r_{\epsilon, \rho}} \\
      &<& C^2 \epsilon^{2r}
  \endeqarray
  for all $p\in M$, all $\epsilon \in \brackets{0,1}$ and all states~$\rho$. Here we used \cref{EqNormState1}, the almost parallelity of~$\ue$ and $r\coloneqq e^{-\brackets{n-2}\ln\brackets{2}/2}<r_{\epsilon, \rho}$. This would yield immediately
  \begin{equation}
    \norm{\nabla \ue}^2_p
    = \max_{\rho \text{ state}} \rho\brackets[\big]{\abs{\nabla \ue}^2_p}
    \overset{\text{(\ref{MoserIterationForf})}}{<} 
    C^2 \epsilon^{2r}
  \end{equation}
  for all $p\in M$, and the lemma would be proved. It remains to show \cref{ConditionMoser}. Fix for the rest of the proof a point ${p\in M}$ with~$f\brackets{p}\neq 0$, $\epsilon \in \brackets{0,1}$, a state~$\rho$ and write~$f$ for~$f_{\epsilon, \rho}$. We calculate $\Laplace f^2\brackets{p}$ in two different ways:
  \begin{align} \label{Laplacef2TwoWays}
    \begin{split}
      \bullet \enspace 
        \Laplace f^2\brackets{p} 
        &= \rho \brackets{\Laplace \scalarproduct{\nabla \ue}{\nabla \ue}_p} \\
        &= \rho \brackets[\Big]{\scalarproduct{\nabla^*\nabla \brackets{\nabla \ue}}{\nabla \ue}_p+\scalarproduct{\nabla \ue}{\nabla^*\nabla \brackets{\nabla \ue}}_p} - 2\rho \brackets{\abs{\nabla \nabla \ue}^2_p} \\
      \bullet \enspace 
        \Laplace f^2\brackets{p}
        &= 2f\brackets{p}\Laplace f\brackets{p}-2\norm{\nabla f}_p^2 
    \end{split}
  \end{align}
  Here we used that~$\rho$ commutes with the Laplacian and the connection. Note that
  \begin{align}\label{HelpEstimateGeneralCstar}
    &\rho \brackets{a^*+a}^2= \brackets{\overline{\rho \brackets{a}}+\rho \brackets{a}}^2 \leq 4 \overline{\rho \brackets{a}}\rho \brackets{a}=4\rho\brackets{a^*}\rho\brackets{a} \quad \text{and}\\ \label{CSStates}
    &\rho\brackets[\big]{\scalarproduct{x}{y}}\rho\brackets[\big]{\scalarproduct{y}{x}}\leq \rho\brackets[\big]{\scalarproduct{x}{x}}\rho\brackets[\big]{\scalarproduct{y}{y}} 
  \end{align} 
  hold for all $a,b\in \A$ and all $x,y$ in a Hilbert $\A$-module~$P$ with inner product $\scalarproduct{\placeholder}{\placeholder}$. The estimate in (\ref{CSStates}) is the Cauchy-Schwarz inequality of the semi-inner product space $P$ with $\rho\brackets{\scalarproduct{\placeholder}{\placeholder}}$. Let $e_1,\cdots,e_n$ be a local orthonormal frame of~$TM$ around~$p$. We obtain
  \begineqarrayst \label{UseKato}
    4 f^{2}\brackets{p}\norm{\nabla f}_p^2 
    = \sum_{i=1}^n \brackets[\big]{\nabla_i\brackets{f^2\brackets{p}}}^2
    &=& \sum_{i=1}^n \rho \brackets[\big]{\scalarproduct{\nabla_i \nabla \ue}{\nabla \ue}+\scalarproduct{\nabla \ue}{\nabla_i \nabla \ue}}^2 \\
    &\overset{\text{(\ref{HelpEstimateGeneralCstar})}}{\leq}& 4\sum_{i=1}^n \rho \brackets[\big]{\scalarproduct{\nabla_i \nabla \ue}{\nabla \ue}} \rho \brackets[\big]{\scalarproduct{\nabla \ue}{\nabla_i \nabla \ue}} \\
    &\overset{\text{(\ref{CSStates})}}{\leq}& 4 \sum_{i=1}^n \rho\brackets{\abs{\nabla_i \nabla \ue}_p^2} \rho \brackets{\abs{\nabla \ue}_p^2} \\
    &=& 4 f^2\brackets{p} \rho\brackets{\abs{\nabla \nabla \ue}_p^2}
  \endeqarray
  Dividing \cref{UseKato} by $2 f^2\brackets{p}$ gives together with \cref{Laplacef2TwoWays}
  \begin{equation} \label{MainEstimate}
    2 f\brackets{p}\Laplace f\brackets{p} \leq \rho \brackets[\Big]{\scalarproduct{\nabla^*\nabla \brackets{\nabla \ue}}{\nabla \ue}_p+\scalarproduct{\nabla \ue}{\nabla^*\nabla \brackets{\nabla \ue}}_p}.
  \end{equation}
  Since the third covariant derivative of~$\ue$ is uniformly bounded in~$\epsilon \in \brackets{0,1}$ by the almost parallelity of~$\set{\ue}_{\epsilon>0}$, there exist positive constants $C_1$ and $\alpha$, independent of~$\epsilon$, such that 
  \begin{equation} \label{BoundInfinity}
    \rho \brackets{\abs{\nabla^*\nabla \brackets{\nabla \ue}}^2_p}\leq \norm{\nabla^*\nabla \brackets{\nabla \ue}}^2_p 
    \leq C_1 \Inftynorm[\big]{\nabla ^3 \ue} 
    \leq \alpha
  \end{equation}
  holds. Finally, we obtain with (\ref{MainEstimate}), (\ref{HelpEstimateGeneralCstar}), (\ref{CSStates}) and (\ref{BoundInfinity}) 
  \begin{equation} 
      \brackets[\Big]{2f\brackets{p}\Laplace f\brackets{p}}^2
      \leq 4 \rho \brackets{\abs{\nabla ^* \nabla \brackets{\nabla \ue}}^2_p} \rho \brackets[\Big]{\abs{\nabla \ue}^2_p}
      \leq 4 \alpha f^2\brackets{p}. 
  \end{equation}
  This proves \cref{ConditionMoser} and the lemma is proved. 
\end{proof}
\begin{lem} \label{AlmostConstantLem}
  Any family of almost parallel sections of~$\SpinBdl$ is almost constant in the sense of \cref{DefAlmost}. 
\end{lem}
\begin{proof}
	Let $\set{\ue}_{\epsilon>0}$ be a family of almost parallel sections. There exists a constant $C_1>0$ such that for all $g\in W^{1,\infty}\brackets{M,\R}$ the Poincaré inequality 
  \begin{equation}
    \Inftynorm{\overline{g}-g}\leq C_1 \Inftynorm{\nabla g} \quad \text{with} \quad \overline{g}\coloneqq \frac{1}{\vol M}\int_M g \dvol
  \end{equation}
  holds \cite[Theorem 1 Section 5.8]{Evans2010}. Since all $\ue$ are smooth, the function 
  \mapdefonelinenoname{M}{\R^+}{p}{\rho \brackets[\big]{\abs{\ue}^2_p}}
  is an element of $W^{1,\infty}\brackets{M,\R}$ for any state $\rho \colon \A\to \C$ and any $\epsilon\in (0,1)$. Applying the Poincaré inequality to this function gives
  \begin{equation} \label{ApplicationPoincareInequality}
    \abs[\big]{\rho \brackets{\overline{u}_\epsilon}-\rho \brackets[\big]{\abs{\ue}^2_p}} 
    \leq C_1 \Inftynorm[\big]{\nabla \rho \brackets[\big]{\abs{\ue}^2_{\placeholder}}}
    \overset{\text{(\ref{TriangleCS})}}{\leq} 2 C_1 \norm{\ue}_\infty \norm{\nabla \ue}_\infty
  \end{equation}
  for all $p\in M$, all $\epsilon \in \brackets{0,1}$ and all states~$\rho$. The second step follows from
  \begin{equation} \label{TriangleCS}
    \norm[\big]{\nabla \rho\brackets[\big]{\abs{\ue}_{\placeholder}^2}}_q^2 
    \leq 4 \rho\brackets[\big]{\abs{\ue}_q^2}\rho\brackets[\big]{\abs{\nabla \ue}_q^2}
    \leq 4 \norm{\ue}_q^2 \norm{\nabla \ue}_q^2,
  \end{equation}
  which holds by the same considerations as in \cref{UseKato}. There exist, by the almost parallelity of~$\set{\ue}_{\epsilon>0}$ and \cref{NablaUeSmallinInfinity}, positive constants $r$, $C_2$ and $C_3$ such that ${\Inftynorm{\ue} < C_2}$ and $\Inftynorm{\nabla \ue} <C_3 \epsilon^r$ hold for all $\epsilon \in \brackets{0,1}$. This yields
  \begineqarray
    \Anorm[\big]{\overline{u}_\epsilon-\abs{\ue}^2_p} 
    \overset{\text{(\ref{EqNormState1})}}{=} \max_{\rho\text{ state}} \abs[\big]{\rho \brackets{\overline{u}_\epsilon}-\rho \brackets[\big]{\abs{\ue}^2_p}}
    &\overset{\text{(\ref{ApplicationPoincareInequality})}}{\leq}& 2 C_1 \Inftynorm{\ue} \Inftynorm{\nabla \ue} \\
    &<& 2 C_1 C_2 C_3 \epsilon^r
  \endeqarray
  for all $p\in M$ and all $\epsilon \in \brackets{0,1}$, hence the lemma is proved.
\end{proof}
%
%
\section{Non-trivial Rosenberg index and the Ricci curvature} \label{SectionZeroRicci}
In this section, we apply the techniques from \cref{SectionAlmostHarmonicSections} to give a direct proof of the well-known fact that any closed connected spin manifold with non-negative scalar curvature and non-vanishing Rosenberg index is Ricci flat.
\begin{thm} \label{ThmRicciFlat}
	Let $\brackets{M,g}$ be a closed connected spin manifold of dimension~$n$, $\A$ a unital Real $\Cstar$-algebra and $E\rightarrow M$ a Real flat bundle of finitely generated projective Hilbert $\A$-modules. Let~$\DE$ be the induced Dirac operator of the $\ComplexCl_n\tensgr \A$-linear Dirac bundle obtained by twisting the spinor bundle~$\SpinBdl$ by~$E$ as in \cref{ConstructionTwistBundle}. If $\scal\geq0$ and the higher index of $\DE$ is non-zero in $\KO_n\brackets{\A}$, the Ricci curvature vanishes. 
\end{thm}
Applying the previous theorem to the maximal group $\Cstar$-algebra of the fundamental group of~$M$, equipped with its natural Real structure, and the Mishchenko bundle yields: 
\begin{cor} \label{CorRicciFlat}
	Any closed connected spin manifold with $\scal\geq 0$ and non-vanishing Rosenberg index is Ricci flat. 
\end{cor}
Let us take a look at the proof of \cref{CorRicciFlat} in the case where the classical index of the spin Dirac operator~$\Dirac$ is non-zero: There exists a $\Ltwoblanc$-normalized $\Dirac$-harmonic section~$u$ of the spinor bundle~$\SpinBdl$, which is by the Schrödinger--Lichnerowicz formula even parallel and satisfies $\scalarproduct{u\brackets{p}}{u\brackets{p}}_p=\brackets{\vol{M}}^{-1}$ for all~$p\in M$. For a local orthonormal frame $e_1,\dots,e_n$ of~$TM|_U$ and any vector field~$X$ supported in~$U$, we obtain
\begin{equation} \label{ProofClassicalRicciFlatEq}
  \label{ClassicalProofRicc0}
  \Ltwonorm{\Ric \brackets{X}}
  =\sqrt{\vol \brackets{M}}\Ltwonorm{\Ric \brackets{X}\cdot u}
  \overset{\text{(\ref{RicCurvatureRealtion})}}{=} 2 \sqrt{\vol \brackets{M}} \norm[\big]{\sum_j e_j\cdot K_{X,e_j} u}_{\Ltwoblanc}
  =0,
\end{equation}
hence $\Ric \equiv 0$. Here \enquote{$\;\cdot\;$} denotes the Clifford multiplication by vector fields, and~$K$ denotes the curvature tensor of~$\SpinBdl$. We used in the second step
\begin{equation} \label{RicCurvatureRealtion}
  \Ric \brackets{X} \cdot u = -2 \sum_j e_j\cdot K_{X,e_j}u \quad \text{\cite[see][Section 3.1]{Friedrich2000}}.
\end{equation} \medbreak
The situation in \cref{ThmRicciFlat} gives rise to a family of almost $\DE$-harmonic section (\cref{DiracInvertible} and \cref{AlmostHarmonicExistence}). The proof of \cref{ThmRicciFlat} splits into two parts. First, we show that the family of almost $\DE$-harmonic sections is almost parallel by the extreme geometric situation, hence it is almost constant by \cref{AlmostConstantLem}. Second, we conclude that~$M$ is Ricci flat by similar calculations as for the corresponding classical statement explained in the previous paragraph.  
\begin{lem} \label{ParallelitySpecialSituation1}
  Assume a situation as in \cref{ThmRicciFlat}. Then there exists a family $\set{\ue}_{\epsilon >0}$ of almost constant sections. Moreover, $\set{\ue}_{\epsilon >0}$ can be chosen such that there exists a constant $C$ such that $\Ltwonorm[\big]{\nabla^2\ue}< C\sqrt{\epsilon}$ holds for all $\epsilon \in \brackets{0,1}$.
\end{lem}
\begin{proof}
The Dirac operator~$\DE$ is not invertible by \cref{DiracInvertible}, hence there exists by \cref{AlmostHarmonicExistence} a family $\set{\ue}_{\epsilon >0}$ of almost $\DE$-harmonic sections. Using the almost $\DE$-harmonicity gives together with the Schrödinger–Lichnerowicz formula, the $\Ltwoblanc$-self-adjointess of~$\DE$ and $\scal \geq 0$  
\begin{equation}
  \epsilon^2 
  > \Ltwonorm{\DE \ue}^2 
  =\Anorm[\big]{\Ltwoscalarproduct[\big]{\DE^2 \ue}{\ue}}
  = \Anorm[\Big]{\Ltwoabs{\nabla \ue}^2 + \frac{1}{4}\int_M \scal_p \abs{\ue}_p^2 \intmathd p}
  \geq \Ltwonorm{\nabla \ue}^2
\end{equation}
for all $\epsilon >0$. \cref{AlmostHarmonicImpliesAlmostParallel} gives that $\set{\ue}_{\epsilon >0}$ is almost parallel, hence almost constant by \cref{AlmostConstantLem}. For the proof of the second part, we fix an $\epsilon \in \brackets{0,1}$. By a commutator formula for $\nabla$ and $\nabla^*\nabla$ given in \cite[Lemma A.3]{Ammann2007}, we obtain
\begin{equation} \label{CommutatorFormula}
   \brackets{\nabla^*\nabla}\nabla\ue =
  \nabla\brackets{\nabla^*\nabla}\ue + \brackets{\Ric \tens \id}\nabla \ue + 2c_{12} \brackets{\id \tens K} \nabla \ue + c_{12}\brackets{\nabla K}\ue.
\end{equation}
Here~$K$ denotes the curvature tensor of the bundle~$\SE$, $c_{12}$ the metric contraction in the first and second slot, and $\Ric$ acts on a 1-form via pre-composition. Let $\scal_{\text{max}}$ be the maximum of the scalar curvature of~$M$. We obtain, by \cref{CommutatorFormula} together with the almost $\DE$-harmonicity, 
\begineqarray 
  \Ltwonorm{\nabla \nabla \ue}^2 
  &=& \Anorm{\Ltwoscalarproduct{\brackets{\nabla^*\nabla}\nabla \ue}{\nabla \ue}} \\
  &\overset{\text{(\ref{CommutatorFormula})}}{\leq}& 
  \Underbrace{\Ltwonorm{\nabla^*\nabla \ue}^2}{\overset{(\star)}{<} \brackets[\big]{1+\tfrac{1}{2}\sqrt{\scal_{\text{max}}}}^2\epsilon^2}
  + \Underbrace{\Anorm{\Ltwoscalarproduct{\brackets{\Ric \tens \id}\nabla \ue}{\nabla \ue}}}{\overset{(\star\star)}{\leq} C_1\Ltwonorm{\nabla \ue}^2} \\
  && + \Underbrace{\Anorm{\Ltwoscalarproduct{2c_{12}\brackets{\id \tens K}\nabla \ue}{\nabla \ue}}}{\overset{(\star\star)}{\leq} C_1\Ltwonorm{\nabla \ue}^2}
  + \Underbrace{\Anorm{\Ltwoscalarproduct{c_{12}\brackets{\nabla K}\ue}{\nabla \ue}}}{\overset{(\star\star)}{\leq} C_1\Ltwonorm{\nabla \ue}\Ltwonorm{\ue}} \\
  &<& C^2 \epsilon
\endeqarray 
for suitable constants $C_1,C>0$, which are both independent of~$\epsilon$. It remains to show the estimates $(\star)$ and $(\star\star)$. \medbreak
$(\star)$ We obtain, by the Schrödinger–Lichnerowicz formula, $\scal \geq 0$ and the almost $\DE$-harmonicity of the section~$\ue$,
  \begin{align} \label{FirstEstimateNablaNabla}
    &\Ltwoabs{\DE \ue}^2 
    = \Ltwoabs{\nabla \ue}^2+\textstyle\frac{1}{4}\displaystyle\Ltwoscalarproduct{\scal\cdot \ue}{\ue}
    \geq \textstyle\frac{1}{4}\displaystyle \Ltwoscalarproduct{\scal \cdot \ue}{\ue}
    \geq 0 \text{ and} \\ \label{SecondEstimateNablaNabla}
    &\Ltwonorm{\scal \cdot \ue}^2 
    \leq \scal_{\text{max}} \Anorm[\Big]{\Underbrace{\int_M \scal_p \abs{\ue}^2_p \intmathd p}{=\Ltwoscalarproduct{\scal \cdot \ue}{\ue}}}
    \overset{\text{(\ref{FirstEstimateNablaNabla})}}{\leq} 4\scal_{\text{max}} \Ltwonorm{\DE \ue}^2
    < 4\scal_{\text{max}} \epsilon^2.
  \end{align}
  Equation $(\star)$ follows by the Schrödinger–Lichnerowicz formula, the triangle inequality, \cref{SecondEstimateNablaNabla}, the $\DE$-harmonicity of~$\ue$, and $\epsilon \in \brackets{0,1}:$
  \begin{equation}
    \Ltwonorm{\nabla^*\nabla \ue}
    = \Ltwonorm[\big]{\DE^2 \ue -\textstyle\frac{1}{4}\displaystyle\scal\cdot \ue}
    \leq \Ltwonorm[\big]{\DE^2 \ue} + \tfrac{1}{4}\displaystyle \Ltwonorm{\scal \cdot \ue}
    < \brackets{1+\tfrac{1}{2}\sqrt{\scal_{\text{max}}}}\epsilon.
  \end{equation} \medbreak
$(\star\star)$ The maps 
  \begin{align}
    &\Ric \tens \id, c_{12} \brackets{\id \tens K} \colon \Cinfty{M,T^*M\tens\SE} \to \Cinfty{M,T^*M\tens\SE} \text{ and} \\
    &c_{12}\brackets{\nabla K} \colon \Cinfty{M,\SE} \to \Cinfty{M,T^*M\tens \SE}
  \end{align}
  decent fiberwise to adjointable maps between the associated Hilbert $\A$-modules for all $p\in M$, hence they extend
  to adjointable maps between the corresponding $\Ltwoblanc$-Hilbert $\A$-modules. By the Cauchy-Schwarz inequality, \cite[Proposition 1.2]{Lance1995} and the almost $\DE$-parallelity of the section $\ue$, we obtain for a suitable constant~$C_1$ the estimates in $(\star\star)$, and the lemma is proved. 
\end{proof}
\begin{proof}[Proof (\cref{ThmRicciFlat})]
	Fix an arbitrary vector field $X$ of the manifold~$M$ with support in a chart~$U$. Let $\set{\ue}_{\epsilon>0}$ be a family of almost constant sections as in \cref{ParallelitySpecialSituation1}. There exists a constant $C>0$ such that 
  \begin{equation} \label{EquationRicciFlat}
    \Ltwonorm{\Ric \brackets{X}} \overset{(\star)}{\leq} 2\sqrt{\vol M} \Ltwonorm{\Ric \brackets{X}\cdot \ue} \overset{(\star\star)}{<} C \sqrt{\epsilon}
  \end{equation}
  holds\footnote{See the analogy of \cref{EquationRicciFlat} to the classical calculation in \cref{ProofClassicalRicciFlatEq}.} for all sufficiently small $\epsilon >0$. Equation (\ref{EquationRicciFlat}) yields $\Ric \brackets{X}=0$ by taking $\epsilon \to 0$, hence~$M$ is Ricci flat. It remains to show $(\star)$ and $(\star\star)$.\medbreak
  (\hypertarget{EqStarProfZeroRicci}{$\star$}) 
    Since the family~$\set{\ue}_{\epsilon>0}$ is almost constant, there exist positive constants $C_1$ and $r$ such that
    \begin{equation} \label{EstimateUeCloseInProof}
      \Anorm[\big]{\overline{u}_\epsilon-\abs{\ue}_p^2}< C_1 \epsilon^r
    \end{equation}
    for all $p\in M$ and all $\epsilon \in \brackets{0,1}$. Fix an $\epsilon\in \brackets{0,1}$ with $4C_1 \vol M \epsilon^r \leq 3$, which is equivalent to 
    \begin{equation} \label{EpsilonSmallConditionEquivalence}
      \frac{1}{\vol M} - C_1 \epsilon^r \geq \frac{1}{4\vol M}.
    \end{equation}
    Multiplying both sides of \cref{EstimateUeCloseInProof} with $\norm{\Ric\brackets{X}}_p^2$ and integrating over $M$ gives
    \begin{equation} \label{UeAlomostConstant1}
      \int_M \norm{\Ric\brackets{X}}_p^2 \Anorm[\big]{\overline{u}_\epsilon-\abs{\ue}_p^2} \intmathd p \leq C_1\epsilon^r \Ltwonorm{\Ric\brackets{X}}^2.
    \end{equation}
    We estimate the left-hand side of \cref{UeAlomostConstant1} by taking the $\A$-norm out of the integral and using the reversed triangle inequality as follows:
    \begineqarrayst \label{UeAlomostConstant2}
        \int_M \norm{\Ric\brackets{X}}_p^2 \Anorm[\big]{\overline{u}_\epsilon-\abs{\ue}_p^2} \intmathd p 
        &\geq& \Anorm[\Big]{\int_M \norm{\Ric\brackets{X}}_p^2 \overline{u}_\epsilon \intmathd p-\int_M \norm{\Ric\brackets{X}}_p^2 \abs{\ue}_p^2 \intmathd p} \\
        &\geq& \abs{\Anorm[\Big]{\overline{u}_\epsilon \int_M \norm{\Ric\brackets{X}}_p^2  \intmathd p} - \Anorm[\Big]{\int_M \norm{\Ric\brackets{X}}_p^2 \abs{\ue}_p^2 \intmathd p}} \\
        &=& \abs[\Big]{\Underbrace{\frac{1}{\vol M}\Ltwonorm{\Ric \brackets{X}}^2-\Ltwonorm{\Ric \brackets{X}\cdot \ue}^2}{\eqqcolon \Theta}}
    \endeqarray
    In the last step, we used $\Anorm{\overline{u}_\epsilon}=\brackets{\vol M}^{-1}$, which holds by the $\Ltwoblanc$-normalization of~$\ue$ and the definition of $\overline{u}_\epsilon$ (\cref{DefAlmost}). Using that~$\epsilon$ satisfies \cref{EpsilonSmallConditionEquivalence} gives, by combining \cref{UeAlomostConstant1} with \cref{UeAlomostConstant2},
    \begin{equation}
      \Ltwonorm{\Ric \brackets{X}\cdot \ue}^2 \geq 
        \left\{\!\begin{aligned}
        &\brackets{\frac{1}{\vol M}-C_1 \epsilon^r} \Ltwonorm{\Ric \brackets{X}}^2 &\text{ if } \Theta >0,\\[1ex]
        &\frac{1}{\vol M} \Ltwonorm{\Ric \brackets{X}}^2 &\text{ if } \Theta \leq 0,\\[1ex]
        \end{aligned}\right\}
      \geq \frac{1}{4\vol M} \Ltwonorm{\Ric \brackets{X}}^2, 
    \end{equation}
    hence equation $(\star)$ is proved. \medbreak
  $(\star\star)$ We fix a local orthonormal frame $e_1,\dots, e_n$ of $TM|_U$, and we use the convention that indices $i$ and $j$ always run over $\set{1,\dots,n}$. The starting point for the proof of equation $(\star\star)$ is, as in the classical case, the formula in \cref{RicCurvatureRealtion} that relates the Ricci tensor with the curvature tensor of~$\SE$. For $X=\sum_i a^i e_i$ and $a\coloneqq \max_{i} \Inftynorm{a^i}$, we obtain 
  \begin{equation} \label{FirstPartOfStarStarProof}
    \Ltwonorm{\Ric \brackets{X}\cdot \ue} 
    = 2 \Ltwonorm[\Big]{\sum_{i,j} a^j e_j\cdot K_{e_i,e_j}\ue} 
    \leq 2n^2 a \max_{i,j} \Ltwonorm[\big]{e_j\cdot K_{e_i, e_j} \ue}.
  \end{equation}
  By \cref{ParallelitySpecialSituation1}, there exists a constant $C_2>0$ such that 
  \begin{align} \label{SecondPartOfStarStarProof}
    \begin{split}
      \Ltwonorm[\big]{e_j\cdot K_{e_i, e_j} \ue}
      &= \Ltwonorm[\big]{K_{e_i,e_j}\ue} \\
      &\leq \Ltwonorm{\nabla_{i}\nabla_{j}\ue}+\Ltwonorm{\nabla_{j}\nabla_{i}\ue}+\Ltwonorm[\big]{\nabla_{\sqbrackets{i,j}}\ue} \\
      &< C_2 \sqrt{\epsilon}
    \end{split}
  \end{align}
  holds for all $\epsilon \in \brackets{0,1}$ and all $i,j$. Equation $(\star\star)$ follows by combining \cref{FirstPartOfStarStarProof} with \cref{SecondPartOfStarStarProof}. This proves \cref{ThmRicciFlat}. 
\end{proof}
\section{Extremality and rigidity for scalar curvature} \label{SectionExtensionGoetteSemmelmann}
Let $\brackets{M,g}$ and $\brackets{N,\overline{g}}$ be two closed connected Riemannian manifolds of dimension~$m$ and~$n$, respectively, and $f\colon M\to N$ a smooth \textit{area non-increasing spin map}, i.e.\ that
\begin{align}
  &\bullet\; \norm{v\wedge w}_p \geq \norm{f_* v \wedge f_* w}_{f\brackets{p}} \quad \forall p\in M \;\forall v,w\in T_pM \quad \text{and} \tag{\text{area non-increasing}}\\
  &\bullet \; w_1\brackets{TM}=f^*w_1\brackets{TN}, \quad w_2\brackets{TM}=f^*w_2\brackets{TN} \tag{\text{spin}}
\end{align}
holds. Here we use the induced norm on $\bigwedge^2 TM$ and $\bigwedge^2 TN$, and denote by~$w_1$ and~$w_2$ the first and second Stiefel--Whitney class, respectively. We equip the vector bundle $TM \oplus f^*TN$ with the metric~$h=g\oplus \brackets{-f^*\overline{g}}$ of signature $\brackets{m,n}$. By the naturality of the Stiefel--Whitney class and the preliminaries in \cref{SpinStructureSubsection}, the spin condition on~$f$ is equivalent to the existence of a spin structure on $\brackets{TM \oplus f^*TN,h}$. Fix such a spin structure. The construction in \cref{SpinStructureSubsection} gives rise to a $\ComplexCl_{m,n}$-linear spinor bundle $\SpinBdl\coloneqq \SM \tensgr f^*\SN \to M$. We denote the Clifford multiplications by  
\begin{equation} \label{CliffordMultiplicationsDefEq}
  c\colon TM\to \End\brackets{\SpinBdl}\quad \text{and} \quad
  \overline{c}\colon f^*TN\to \End\brackets{\SpinBdl}.
\end{equation}
Note that $c\brackets{v}$ is skew-adjoint, $c\brackets{v}^2=-g\brackets{v,v}$, $\overline{c}\brackets{w}$ is self-adjoint, and $\overline{c}\brackets{w}^2=\overline{g}\brackets{w,w}$ for all $v\in TM$ and all $w\in f^*TN$. In particular, the $\ComplexCl_{m,n}$-linear spinor bundle is a graded Real $\ComplexCl_{m,n}$-linear Dirac bundle $\SM \tensgr f^*\SN \to M$. Denote the induced Dirac operator by~$\Dirac$. \par
%
The curvature tensor~$R^N$ of the manifold~$N$ induces a self-adjoint curvature operator~$\mathcal{R}_N$ on~$\bigwedge^2 TN$ satisfying locally
\begin{equation} \label{EqDefCurvatureOp}
  \overline{g}\brackets[\big]{\mathcal{R}_N\brackets{\overline{e}_i\wedge \overline{e}_j},\overline{e}_k\wedge \overline{e}_l} = -\overline{g}\brackets[\big]{R^N_{\overline{e}_i,\overline{e}_j}\overline{e}_k,\overline{e}_l}
\end{equation}
for all $i,j,k,l\in \set{1,\dots,n}$ and every local orthonormal frame $\overline{e}_1,\dots,\overline{e}_n$ of~$TN$. \textcite[Lemma 1.6]{Goette2000} showed that $\scal_M\geq \scal_N\op f$ implies $\scal_M=\scal_N \op f$ if the manifolds~$M$ and~$N$ are orientable, the curvature operator~$\mathcal{R}_N$ is non-negative, and $\SM \tensgr f^*\SN$ admits a non-zero $\Dirac$-harmonic section. Here non-negativity of the curvature operator means that all its eigenvalues are pointwise non-negative. 
Furthermore, they give a sufficient condition, involving estimates for the scalar and the Ricci curvature of $M$ and $N$, that~$f$ is a Riemannian submersion. \medbreak
In this section, we generalize this result to the situation, where we include the non-orientable case and replace the existence of a $\Dirac$-harmonic section on $\SM \tensgr f^*\SN$ by the existence of a family of almost $\DE$-harmonic sections. Here~$E\to M$ is a Real flat bundle of finitely generated projective Hilbert $\A$-modules for a unital Real $\Cstar$-algebra~$\A$, and $\DE$ denotes the induced Dirac operator of the $\ComplexCl_{m,n}\tensgr\A$-linear Dirac bundle $\SM \tensgr f^*\SN \tensgr E$. The main step in the proof will be that the almost $\DE$-harmonic sections are by the extreme geometric situation assumed in the theorem almost $\DE$-parallel, hence they are almost constant in the sense \cref{DefAlmost}. \par
Furthermore, we will introduce in \cref{HigherMappingDegree} the higher degree of $f$ in analogy to the $\hat{A}$-degree in \cite[Section 2.b]{Goette2000}. It turns out that the cut-and-paste principle for the higher index can be used to prove an index theorem (\cref{IndexThm}) similar to the classical one for the $\hat{A}$-degree \cite[see][Proof of Theorem 2.4]{Goette2000}. This yields that if the product of the higher degree of~$f$ with the Euler characteristic of~$N$ is non-zero, there exists a family of almost $\DE$-harmonic sections. This gives a generalization of \cite[Theorem 2.4]{Goette2000} involving the higher degree of $f$ instead of the $\hat{A}$-degree. \medbreak
Given manifolds~$M$, $N$ and a map $f\colon M\to N$ as in the beginning of this section, we denote the preimage of a regular value~$p$ under~$f$ by~$M_p$, the induced graded Real $\ComplexCl_{m,n}$-linear Dirac bundle $\SM\tensgr f^*\SN$ by~$\SpinBdl$, the induced Dirac operator by~$\Dirac$ and the Clifford multiplications given in \cref{CliffordMultiplicationsDefEq} by~$c$ and $\overline{c}$. Moreover, given a unital Real $\Cstar$-algebra~$\A$ and a Real flat bundle~$E$ of finitely generated projective Hilbert $\A$-modules, we denote the graded Real $\ComplexCl_{m,n}\tensgr\A$-linear Dirac bundle $\SpinBdl \tensgr E$ by~$\SE$ and the induced Dirac operator by~$\DE$. 
\begin{thm} \label{RigidityAlmostHarmonicSpinorThm}
  Let $f\colon M\to N$ be a smooth area non-increasing spin map between two closed connected Riemannian manifolds $\brackets{M,g}$ and $\brackets{N,\overline{g}}$ of dimension~$m$ and~$n$, respectively. Suppose the curvature operator of $N$ is non-negative, and there exists a unital Real $\Cstar$-algebra~$\A$ and a Real flat bundle~$E$ of finitely generated projective Hilbert $\A$-modules together with a family of almost $\DE$-harmonic sections. Then $\scal_M\geq \scal_N\op f$ on $M$ implies $\scal_M=\scal_N \op f$. If, moreover, $\scal_N>2\Ric_N>0$ (or~$f$ is distance non-increasing and $\Ric_N>0$), then $\scal_M\geq \scal_N\op f$ implies that~$f$ is a Riemannian submersion. 
\end{thm}
\begin{proof}
  We adapt the proof of \textcite[Lemma 1.6]{Goette2000}. In the first step, we recap the estimates of~$\DE^2$ using the Lichnerowicz formula for twisted Dirac operators, the non-negativity of the curvature operator of~$N$ and the fact that~$f$ is area non-increasing. In the second step, we show that the family of almost $\DE$-harmonic sections is almost constant, and we prove the extremality and rigidity statements. \par
  Suppose~$\scal_M\geq \scal_N\op f$ on~$M$. Fix open subsets~$U\subset M$ and~$\overline{U}\subset N$ trivializing~$TM$ and~$TN$, respectively, with $f^{-1}\brackets[\normalsize]{\overline{U}}=U$. Let~$e_1,\dots,e_m$ be a $g$-orthonormal frame of~$TU$ and $\overline{e}_1,\dots,\overline{e}_n$ a $\overline{g}$-orthonormal frame of~$T\overline{U}$ such that there exists non-negative $\mu_1,\dots,\mu_{\min\brackets{m,n}}\in \Cinfty{U,\R}$ with 
  \begin{equation}
    f_*e_i=
    \begin{cases}
      \mu_i \overline{e}_i &\text{if } i\leq \min\brackets{m,n} \\
      0 & \text{otherwise}.
    \end{cases}
  \end{equation}
  There exist locally compatible spin structures of~$M$ and~$N$ such that the~$\ComplexCl_{m,n}$-linear Dirac bundle~$\SpinBdl$ coincides with the graded tensor product of~$\SpinBdl M$ with~$f^*\SpinBdl N$, where~$\SpinBdl N$ has typical fiber~$\ComplexCl_{0,n}$. Using the Lichnerowicz formula for twisted spinor bundles (\cite[Theorem 8.17]{Lawson1989}) and the formula for the curvature transform (\cite[Theorem 4.15]{Lawson1989}) gives locally 
  \begin{equation} \label{BLWD2}
    \begin{split}
      \DE^2 
      = \nabla^*\nabla&+\frac{\scal_M}{4}+\brackets[\Bigg]{\Underbrace{\frac{1}{2}\sum_{i,j=1}^m c \brackets{e_i}c \brackets{e_j} \brackets{\id_{\SpinBdl M}\tensgr R_{e_i,e_j}^{f^* \SpinBdl N}}}{= \mathfrak{R}^{f^* \SpinBdl N}}}\tensgr \id_E \\
      \mathfrak{R}^{f^* \SpinBdl N}&= \frac{1}{8} \sum_{i,j=1}^m \sum_{k,l=1}^n \overline{g} \brackets{R_{f_*e_i,f_*e_j}^{N} \overline{e}_k,\overline{e}_l} c\brackets{e_i}c\brackets{e_j}\overline{c}\brackets{f^*\overline{e}_k}\overline{c}\brackets{f^*\overline{e}_l} \\
      &=-\frac{1}{8} \sum_{i,j=1}^m \sum_{k,l=1}^n \overline{g} \brackets{\mathcal{R}_N f_*\brackets{e_i\wedge e_j},\overline{e}_k\wedge \overline{e}_l} c\brackets{e_i}c\brackets{e_j}\overline{c}\brackets{f^*\overline{e}_k}\overline{c}\brackets{f^*\overline{e}_l}.
    \end{split}
  \end{equation}
  Here $R^{f^*\SpinBdl N}$ denotes the curvature tensor of~$f^*\SpinBdl N$ and $\mathcal{R}_N$ the curvature operator of~$N$ on~$\bigwedge^2TN$ as defined in \cref{EqDefCurvatureOp}. Note that the different behaviors of~$c$ and~$\overline{c}$ do not influence the calculations in \cite{Lawson1989}, since all calculations are in an orthonormal frame. By the non-negativity of the curvature operator~$\mathcal{R}_N$, there exists a self-adjoint square root~$L\in\End\brackets[\big]{\bigwedge^2 TN}$ satisfying
  \begin{equation}
    g \brackets[\big]{L\brackets{\overline{e}_i\wedge \overline{e}_j},L\brackets{\overline{e}_k\wedge \overline{e}_l}} = g \brackets[\big]{\mathcal{R}_N \brackets{\overline{e}_i\wedge \overline{e}_j},\overline{e}_k\wedge \overline{e}_l}
  \end{equation}
  for all $i,j,k,l\in \set{1,\cdots,n}$. The same estimates as in \cite[p.5-6]{Goette2000} yield locally
  \begin{equation} \label{EstimateRestTermBLW}
    \begin{aligned} 
      &\frac{\scal_M}{4}-\frac{1}{8} \sum_{i,j=1}^m \sum_{k,l=1}^n \overline{g} \brackets{\mathcal{R}_N f_*\brackets{e_i\wedge e_j},\overline{e}_k\wedge \overline{e}_l} c\brackets{e_i}c\brackets{e_j}\overline{c}\brackets{f^*\overline{e}_k}\overline{c}\brackets{f^*\overline{e}_l} 
      && \\
      &\hspace{2cm} \geq \frac{\scal_M}{4}-\frac{\scal_N\op f}{8}-\frac{1}{8}\sum_{i,j=1}^n\mu^2_i\mu^2_j\brackets[\big]{R_{ijji}^N\op f}
      &\eqqcolon& h_1 \\
      &\hspace{2cm} \geq \frac{\scal_M-\scal_N \op f}{4}
      &\eqqcolon& h_2 \\
      &\hspace{2cm} \geq 0.&&
    \end{aligned}
  \end{equation}
  Let~$V$ be an open subset of~$U$ such that there exists a smooth cut-off function $\psi\colon M\to \R$, which is identically one on~$V$, zero outside~$U$ and $\image \brackets{\psi} \subset \sqbrackets{0,1}$. Then \cref{EstimateRestTermBLW} still holds for~$h_1$ and~$h_2$ replaced by~$\psi h_1$ and~$\psi h_2$, respectively, and we obtain together with \cref{BLWD2} globally 
  \begin{equation} \label{EstimateDsqGlobally}
    \DE^2\geq \nabla^*\nabla +\psi h_1\geq \nabla^*\nabla +\psi h_2\geq \nabla^*\nabla.
  \end{equation}
  By assumption, there exists a family~$\{\ue\}_{\epsilon>0}$ of almost $\DE$-harmonic sections, which satisfy by \cref{EstimateDsqGlobally} 
  \begin{equation} \label{Parallelityofue}
    \Ltwoabs{\DE \ue}^2 
    \geq \Ltwoabs{\nabla \ue}^2+\Ltwoscalarproduct{\psi h_1\cdot\ue}{\ue}
    \geq \Ltwoabs{\nabla \ue}^2+\Ltwoscalarproduct{\psi h_2\cdot \ue}{\ue}
    \geq \Ltwoabs{\nabla \ue}^2
  \end{equation}
  for all~$\epsilon>0$. It follows, by the almost $\DE$-harmonicity, \cref{Parallelityofue} and \cref{AlmostHarmonicImpliesAlmostParallel}, that~$\{\ue\}_{\epsilon>0}$ is almost parallel, hence almost constant by \cref{AlmostConstantLem}. By similar calculations as in the proof of \cref{ThmRicciFlat} part~(\hyperlink{EqStarProfZeroRicci}{$\star$}), we obtain 
  \begin{equation}
    \norm{\psi h_i}_{\Lblanc^1}
    \leq 2 \vol \brackets{M} \Anorm{\Ltwoscalarproduct{\psi h_i \ue}{\ue}}
    \overset{\text{(\ref{Parallelityofue})}}{\leq} 2 \vol \brackets{M} \Ltwonorm{\DE \ue}^2
    < 2 \vol \brackets{M} \epsilon^2
  \end{equation}
  for all sufficiently small~$\epsilon$ and $i=1,2$. This yields in the limit~$\epsilon \to 0$ 
  \begin{align} \label{Rigidity1}
    &h_1 = \frac{\scal_M}{4}-\frac{\scal_N\op f}{8}-\frac{1}{8}\sum_{i,j=1}^n\mu^2_i\mu^2_j\brackets[\big]{R_{ijji}^N\op f}
    \equiv 0 \text{ and}\\
    &h_2 = \frac{\scal_M-\scal_N \op f}{4} \label{Rigidity2}\equiv 0
  \end{align}
  on~$V$, hence the extremality statement is shown by \cref{Rigidity2}. Furthermore, \cref{Rigidity1} together with $\scal_M=\scal_N\op f$ implies $\sum_{i,j=1}^n\mu^2_i\mu^2_j\brackets[\big]{R_{ijji}^N\op f}=\scal_N\op f$ on~$V$, and the rigidity statement follows as in \cite[Section 1.c]{Goette2000}. 
\end{proof}
Let $f\colon M\to N$ be as in \cref{RigidityAlmostHarmonicSpinorThm}, and assume additionally that both manifolds are orientable and of dimension $m=n+k$ and~$n$, respectively. The $\hat{A}$-degree of~$f$ is given by 
\begin{equation}
  \Adeg\brackets{f}\coloneqq \brackets[\big]{\hat{A}\brackets{M}f^*\omega} \sqbrackets{M},
\end{equation}
where $\omega\in H^n\brackets{N;\Z}$ is the cohomological fundamental class of~$N$ corresponding to the orientation on~$N$, $\hat{A}\brackets{M}\in H^*\brackets{M;\Q}$ the $\hat{A}$-class of~$M$ and $\sqbrackets{M}\in H_m\brackets{M;Z}$ the fundamental class of~$M$ corresponding to its orientation. One sufficient condition for the existence of a non-trivial $\Dirac$-harmonic section is that the Euler characteristic of~$N$ and the $\hat{A}$-degree of~$f$ do not vanish \cite[Section 2.b]{Goette2000}. This follows by 
\begin{equation} \label{ClassicConnectionIndexADegree}
  \ind \Dirac =  \chi \brackets{N}\cdot\Adeg\brackets{f},
\end{equation}
and the fact that a non-zero classical index implies the existence of a $\Dirac$-harmonic section. We know, by \cref{DiracInvertible} (2) and \cref{AlmostHarmonicExistence}, that a non-vanishing higher index of~$\DE$ implies the existence of a family of almost $\DE$-harmonic sections. We want to define a higher version of the $\hat{A}$-degree such that \cref{ClassicConnectionIndexADegree} still holds for the higher index of an appropriate twisted Dirac operator.\par
Let us take a closer look at the $\hat{\A}$-degree. Fix a regular value~$p$ of~$f$ and pick a sufficiently small neighborhood~$U_p$ of~$p$ such that there exists an open subset $V_p\subset \R^k$ and an isomorphism $\Phi\colon f^{-1}\brackets{U_p} \to M_p\times V_p$ with $f\vert _{f^{-1}\brackets{U_p}}=\pr_1\op \Phi$. Let~$\omega$ be an $n$-form on~$N$ that represents the cohomological fundamental class of~$N$. Then we can rewrite the $\hat{A}$-degree of~$f$ as follows:
\begin{equation} \label{AhatDegreeRewrittenEq}
  \Adeg\brackets{f}
  =\int_{M_p\times V_p}\hat{A}\brackets{M}f^*\omega
  =\int_{M_p}\hat{A}\brackets{M_p}
  =\ind \brackets[\big]{\Dirac_{\SpinBdl M_p}}
\end{equation}
The last step holds by the Atiyah-Singer index theorem, where~$\Dirac_{\SpinBdl M_p}$ is the spin Dirac operator of the spin manifold $M_p$ (see \cref{LemmaPreimageIsSpin}). This matches with the original notion of $\hat{A}$-degree by \textcite{Gromov1980}. We define in analogy to \cref{AhatDegreeRewrittenEq} the higher mapping degree: 
\begin{defi} \label{HigherMappingDegree}
  Let $f: M\to N$ be a smooth spin map between closed connected manifolds of dimension $n+k$ and $n$, respectively. The \textit{higher degree} of~$f$ is defined via
  \begin{equation}
    \hideg \brackets{f} \coloneqq \ind \brackets[\Big]{\Dirac_{\SpinBdl M_p\tensgr \mathcal{L}\brackets{M}\vert_{M_p}}} \in \KO_k\brackets[\big]{\Cstar{\pi_1\brackets{M}}}
  \end{equation}
  for a regular value~$p$ of~$f$. Here $\SpinBdl M_p$ denotes the spinor bundle of~$M_p$, $\mathcal{L}\brackets{M}$ the Mishchenko bundle of~$M$ and~$\Cstar{\pi_1\brackets{M}}$ the maximal group $\Cstar$-algebra of the fundamental group of~$M$. 
\end{defi}
The higher mapping degree is well-defined by the following lemma and by the bordism invariance of the higher index. 
\begin{lem} \label{LemmaPreimageIsSpin}
  Let $f\colon M\to N$ be a spin map between two closed manifolds and~$p$ a regular value of~$f$. Then the manifold~$M_p$ is orientable and admits a spin structure. 
\end{lem}
\begin{proof}
  Since~$p$ is a regular value of~$f$, the restriction of the tangent bundle to~$M_p$ splits $TM\vert_{M_p}=TM_p\oplus f^*TN\vert_{M_p}$. This gives
  \begin{align}
    &w_1\brackets{TM_p}+w_1\brackets{f^*TN\vert_{M_p}}=w_1\brackets{TM\vert_{M_p}}=w_1\brackets{f^*TN\vert_{M_p}} \quad \text{and} \\
    &w_2\brackets{TM_p}+w_2\brackets{f^*TN\vert_{M_p}}+w_1\brackets{TM_p}w_1\brackets{f^*TN\vert_{M_p}}=w_2\brackets{TM\vert_{M_p}}=w_2\brackets{f^*TN\vert_{M_p}}
  \end{align}
  by the spin property of~$f$ together with the naturality and the Whitney sum formula for the Stiefel--Whitney class. It follows that the first and second Stiefel--Whitney classes of~$TM_p$ vanish, hence the lemma is proved.  
\end{proof}
We will show, by the Poincaré--Hopf lemma and the cut-and-paste principle in \cite[Section 2.3]{Cecchini2020}, that the higher degree of~$f$ shows up in an analogous index formula as in \cref{ClassicConnectionIndexADegree}.
\begin{thm}[Index theorem] \label{IndexThm}
  Let $f\colon M\to N$ be a smooth spin map between two closed connected Riemannian manifolds of dimension~$m=n+k$ and~$n$, respectively, $\A$ a unital Real $\Cstar$-algebra and $E\to M$ a Real flat bundle of finitely generated projective Hilbert $\A$-modules. Then the higher index of the twisted Dirac operator $\DE$ is given by
  \begin{equation}
    \ind\brackets{\DE}=\chi\brackets{N} \cdot \ind \brackets[\Big]{\Dirac_{\SpinBdl M_p\tensgr E\vert_{M_p}}} \in \KO_k\brackets{\A}
  \end{equation}
  for any regular value~$p$ of~$f$. In particular, we obtain for the higher degree of~$f$
  \begin{equation}
    \ind \brackets[\big]{\Dirac_{\mathcal{L\brackets{M}}}}=\chi\brackets{N}\cdot \hideg\brackets{f}\in \KO_k\brackets{\Cstar \pi_1\brackets{M}}.
  \end{equation}
\end{thm}
\begin{proof}
The main idea of the proof is to use the cut-and-paste principle in \cite[Section 2.3]{Cecchini2020} to split off the higher degree of~$f$ from the index of~$\DE$. The remaining factor can be identified with the Euler characteristic of~$N$ via the Poincaré--Hopf lemma if we perform the cut-and-paste principle with respect to a suitable vector field on~$N$. At first, we construct the vector field. Second, we show how this vector field gives rise to a situation in which the cut-and-paste principle is applicable. And finally, we verify the index formula. \par
Let~$\xi$ be a vector field on~$N$ transversal to the zero section, i.e.\ for any zero~$q$ of~$\xi$ and any chart~$\phi$ around~$q$ the differential of the induced map on Euclidean space has full rank at the point~$\phi\brackets{q}$. Such a vector field always exists, since transversality is a generic property. 
It follows that the set of zeros of~$\xi$ is discrete, hence finite by the compactness of~$N$. We deform the vector field without changing the transversality such that all zeros~$p_1,\dots,p_l$ of~$\xi$ are regular values of~$f$, which is possible because almost all points in~$N$ are regular values of~$f$ by Sard's theorem. Fix for all $i\in \set{1,\dots,l}$ a chart
\mapdefoneline{\phi_i}{N\supset U_i}{\phi_i\brackets{U_i}\subset \R^n}{p_i}{0,}
around~$p_i$ such that $\cap_{i=1}^l U_i=\emptyset$.
Fix $R>\tau>r>\epsilon>0$ such that $\Ball_{R+\epsilon}\brackets{0}\subset \phi_i\brackets{U_i}$ for all $i\in \set{1,\dots,l}$. Since~$\xi$ is transversal to the zero section, we can deform the vector field to a transversal vector field~$\xi_1$, which satisfies for suitable non-zero constants $\lambda_{ij}\in \R$
\begin{equation}
  \brackets[\Big]{\mathd_{\phi_i^{-1}\brackets{x}} \phi_i}\brackets[\Big]{\brackets{\xi_1\op \phi_i^{-1}}\brackets{x}} =
  \begin{cases}
    \diag\brackets{\lambda_{i1},\dots,\lambda_{in}}x & \text{if } x\in \Ball_{\epsilon}\brackets{0}, \\[1ex]
    \diag\brackets{\lambda_{i1},\dots,\lambda_{in}}\frac{x}{\norm{x}}\tau & \text{if } x\in \Ball_R\brackets{0}\setminus \Ball_r\brackets{0}
  \end{cases}
\end{equation} 
for all $i\in \set{1,\cdots,l}$ and has the same zeros as~$\xi$. By the Poincaré--Hopf lemma, the Euler characteristic of~$N$ is given by
\begin{equation} \label{EulerCharPoincareHopf}
  \chi\brackets{N}
  =\sum_{i=1}^l \deg\brackets{\xi_1,p_i}
  =\sum_{i=1}^l \sgn\brackets{\det \diag\brackets{\lambda_{i1},\dots,\lambda_{in}}} 
  = \sum_{i=1}^l \prod_{j=1}^n \sgn\brackets{\lambda_{ij}}.
\end{equation}
Here $\deg\brackets{\xi_1,p_i}$ denotes the local Poincaré--Hopf index of~$\xi_1$ at~$p_i$ \cite[Section 1.1]{Brasselet2009}. In the next step, we define the manifolds and bundles involved in the cut-and-paste process (see \cref{FigManifoldsCutAndPast}). 
\begin{figure} 
  \centering
  \includegraphics[width=\textwidth]{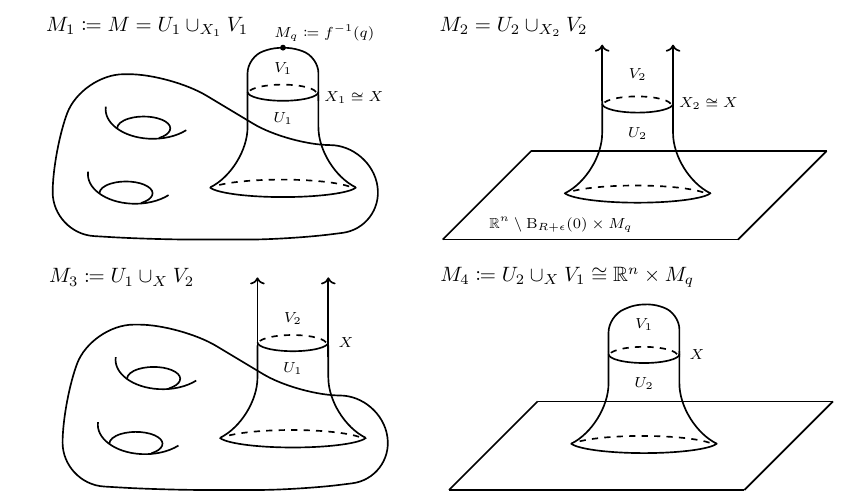} 
  \caption{An illustration of the manifolds involved in the cut-and-paste principle in the proof of the index theorem. Here the vector field~$\xi_1$ has just one zero $q\in N$.}
  \label{FigManifoldsCutAndPast}
\end{figure}
Since all~$p_i$ are regular values of~$f$, there exists for any $i\in\{1,\dots,l\}$ an isomorphism $\psi_i\colon \brackets{\phi_i\op f}^{-1}\brackets[\big]{\Ball_R\brackets{0}} \to \Ball_R\brackets{0}\times M_{p_i}$ compatible to~$\phi_i$ in the sense that 
\begin{center} \begin{tikzcd}
  \brackets{\phi_i\op f}^{-1}\brackets[\big]{\Ball_R\brackets{0}} \arrow{r}{f} \arrow[swap]{d}{\psi_i} \arrow{d}{\cong} & \phi_i^{-1}\brackets[\big]{\Ball_R\brackets{0}} \arrow{d}{\phi_i} \arrow[swap]{d}{\cong} \\
\Ball_R\brackets{0}\times M_{p_i} \arrow{r}{\pr_1} & \Ball_R\brackets{0}
\end{tikzcd} \end{center}
commutes. We define 
\begin{align} 
  M_1 \coloneqq & M = U_1 \cup_{X_1} V_1 \text{ with} \\
    & U_1 \coloneqq M\setminus \bigsqcup_{i=1}^l \brackets{\phi_i\op f}^{-1}\brackets[\big]{\interior{\Ball_\tau\brackets{0}}}, \quad 
    X_1 \coloneqq \bigsqcup_{i=1}^l \partial \brackets{\phi_i\op f}^{-1}\brackets[\big]{\Ball_\tau\brackets{0}} \text{ and } \\
    &V_1 \coloneqq \bigsqcup_{i=1}^l \brackets{\phi_i\op f}^{-1}\brackets[\big]{\Ball_\tau\brackets{0}},
   \\
  M_2 \coloneqq & \bigsqcup_{i=1}^l \brackets[\Big]{\underbrace{\R^n\setminus \interior{\Ball_r\brackets{0}}\cup_{\partial \Ball_r\brackets{0}}\partial \Ball_r\brackets{0}\times[0,\infty)}_{\cong S^{n-1}\times \R \eqqcolon N_2^i}}\times M_{p_i}=U_2\cup_{X_2}V_2 \text{ with} \\
    & U_2\coloneqq \bigsqcup_{i=1}^l \R^n\setminus \interior{\Ball_\tau \brackets{0}}\times M_{p_i},\quad X_2\coloneqq \bigsqcup_{i=1}^l \partial \Ball_\tau\brackets{0}\times M_{p_i} \text{ and}\\
    & V_2\coloneqq \bigsqcup_{i=1}^l \brackets[\Big]{\Ball_\tau\brackets{0}\setminus \interior{\Ball_r\brackets{0}}\cup_{\partial \Ball_r\brackets{0}}\partial\Ball_r\brackets{0}\times [0,\infty)}\times M_{p_i}.
\end{align}
Taking the union of the maps~$\psi_i$ and restrict properly gives an isomorphism between neighborhoods of~$X_1$ and~$X_2$
\begin{equation}
  \Gamma \colon
  \underbrace{\bigsqcup_{i=1}^l \brackets{\phi_i\op f}^{-1}\brackets[\big]{\interior{\Ball_R\brackets{0}}} \setminus \brackets{\phi_i\op f}^{-1}\brackets[\big]{\Ball_r\brackets{0}}}_{\eqqcolon \;\mathcal{U}\brackets{X_1}} \To{}
  \underbrace{\bigsqcup_{i=1}^l \interior{\Ball_R\brackets{0}}\setminus \Ball_r\brackets{0}\times M_{p_i}}_{\eqqcolon \;\mathcal{U}\brackets{X_2}},
\end{equation}
which restricts to a diffeomorphism from~$X_1$ to~$X_2$. Equip~$M_1, M_2$ and~$N$ with metrics turning~$\mathcal{U}\brackets{X_1}, \mathcal{U}\brackets{X_2}$ and $\bigsqcup_{i=1}^l \phi_i^{-1}\brackets[\big]{\interior{\Ball_R\brackets{0}}\setminus \Ball_r\brackets{0}}$ into tubular neighborhoods of $X_1, X_2$ and $\bigsqcup_{i=1}^l \partial \phi_i^{-1}\brackets[\big]{\Ball_\tau \brackets{0}}$, respectively, turning~$\Gamma$ into an isometry, and carrying on the part $\bigsqcup_{i=1}^l \R^n \setminus \Ball_{R+\epsilon}\brackets{0}\times M_{p_i}$ the product metric of the Euclidean metric and the induced metric on~$M_{p_i}$. We denote for $i\in \set{1,\dots,l}$ the projections of~$N_2^i\times M_{p_i}$ to the first and second factor by~$\pr_1^i$ and~$\pr_2^i$, respectively, and define 
\begin{equation}
  \pr_1\coloneqq \bigsqcup_{i=1}^l \pr_1^i\colon M_2 \To{} N_2 \coloneqq \bigsqcup_{i=1}^l N_2^i.
\end{equation}
Furthermore, we obtain by radial parallel transport an isomorphism
\begin{equation}
  \bigsqcup_{i=1}^l \brackets[\big]{\pr_2^i \op \psi_i \vert_{\mathcal{U}\brackets{X_1}}}^* E\vert_{M_{p_i}}\cong E\vert_{\mathcal{U}\brackets{X_1}}
\end{equation}
respecting the bundle metric and the connection. Consider the graded Real $\ComplexCl_{m,n}\tensgr \A$-linear Dirac bundles
\begin{align} \begin{split}
  &W_1\coloneqq \SE = SM \tensgr f^*SN\tensgr E \to M=M_1 \text{ and} \\
  &W_2\coloneqq SM_2\tensgr \bigsqcup_{i=1}^l \brackets[\big]{\pr_1^i}^*SN_2^i\tensgr \bigsqcup_{i=1}^l\brackets[\big]{\pr_2^i}^*E\vert_{M_{p_i}}\to M_2
\end{split} \end{align}
together with their induced Dirac operators~$\Dirac_1$ and~$\Dirac_2$, respectively. Here we consider $\SpinBdl N_2^i$ as the $\ComplexCl_{0,n}$-linear spinor bundle of~$N_2^i$. The bundle~$W_2$ comes by construction with a self-adjoint Clifford multiplication by the pullback of tangent vectors of~$N_2$ along~$\pr_1$, which we denote again by~$\overline{c}$. Furthermore, we define for $\lambda>0$ perturbed Dirac operators
\begin{equation}
  \mathrm{B}_1\coloneqq \Dirac_1+\lambda \overline{c}\brackets{f^*\xi_1} \text{ and } \mathrm{B}_2\coloneqq \Dirac_2+\lambda \overline{c}\brackets[\big]{\pr_1^*\xi_2}.
\end{equation}
Here $\xi_2$ denotes a nowhere vanishing vector field on $N_2$, which satisfies
\begin{equation}
  \xi_2\brackets{x}=
  \begin{cases}
    \diag\brackets{\lambda_{i1},\dots,\lambda_{in}}\frac{x}{\norm{x}}\tau & \text{if } x\in \Ball_R\brackets{0}\setminus \Ball_r\brackets{0}\subset N^i_2\subset N_2, \\[1ex]
    \diag\brackets{\lambda_{i1},\dots,\lambda_{in}}x & \text{if } x\in \R^n\setminus{\Ball_{R+\epsilon}\brackets{0}}\subset N^i_2 \subset N_2
  \end{cases}
\end{equation}
for all $i\in \set{1,\dots,l}$. The unbounded operators~$\mathrm{B}_1$ and~$\mathrm{B}_2$ are odd Real formally self-adjoint elliptic differential operators of first order. Here it is crucial that~$\overline{c}$ is self-adjoint and not skew-adjoint\footnote{This justifies that we consider the spinor bundle associated to $TM\oplus f^*TN$ equipped with the indefinite metric of signature~$\brackets{m,n}$ instead of the induced positive definite metric.}. By construction, there is a bundle isometry $\widetilde{\Gamma}\colon W_1\vert_{\mathcal{U}\brackets{N_1}} \to W_2\vert_{\mathcal{U}\brackets{N_2}}$ covering~$\Gamma$ and satisfying $\mathrm{B}_2=\widetilde{\Gamma}\op \mathrm{B}_1\op \widetilde{\Gamma}^{-1}$. Fix a manifold~$X$ isomorphic to $X_1\cong X_2$. Now we perform the cut-and-paste procedure and obtain graded Real $\ComplexCl_{m,n}\tensgr \A$-linear Dirac bundles
\begin{align} \begin{split}
  & W_3 \coloneqq W_1\vert_{U_1} \cup_X W_2\vert_{V_2} \to M_3\coloneqq U_1 \cup_X V_2 \text{ and}\\
  & W_4 \coloneqq W_2\vert_{U_2} \cup_X W_1\vert_{V_1} \to M_4\coloneqq U_2 \cup_X V_1
\end{split} \end{align}
with induced Dirac operators~$\Dirac_3$ and~$\Dirac_4$, together with perturbed Dirac operators
\begin{equation}
  \mathrm{B}_3\coloneqq \Dirac_3+\lambda \overline{c}\brackets{\xi_3} \text{ and } \mathrm{B}_4\coloneqq \Dirac_4+\lambda \overline{c}\brackets[\big]{\xi_4}.
\end{equation}
Here $\xi_3$ and $\xi_4$ are vector fields on~$M_3$ and~$M_4$, respectively, obtained from~$f^*\xi_1$ and~$\pr_1^*\xi_2$ via the cut-and-paste procedure. If~$\mathrm{B}_2^2$ is uniformly positive at infinity, the index formula in \cite[Theorem 2.12]{Cecchini2020} gives
\begin{equation} \label{CuttingIndexformula}
  \ind \mathrm{B}_1+\ind \mathrm{B}_2=\ind \mathrm{B}_3+\ind \mathrm{B}_4\in \KO_k\brackets{\A}.
\end{equation}
It remains to show that for a regular value~$p$ of~$f$ and a sufficiently big $\lambda>0$ the following holds:
\begin{enumerate}
  \item $\ind \DE=\ind \mathrm{B}_1$ 
  \item The operators $\mathrm{B}_2^2$ and $\mathrm{B}_3^2$ are uniformly positive and hence $\ind \mathrm{B}_2=\ind \mathrm{B}_3=0$. 
  \item $\ind \mathrm{B}_4 = \chi\brackets{N} \cdot \ind\brackets[\Big]{\Dirac_{SM_p\tensgr E\vert_{M_p}}}$ 
\end{enumerate} \medbreak
\textbf{Ad(1)} Since the graded Real $\ComplexCl_{m,n}\tensgr \A$-linear Dirac bundles $W_1\to M_1$ and $\SE\to M$ coincide except the involved metrics, the indices of the induced Dirac operators match. The index of the perturbed Dirac operator~$\mathrm{B}_1$ matches with the index of~$\DE$ by \cite[Proposition 4.1]{Cecchini2020b}. \medbreak
\textbf{Ad(2)} We show first that $\mathrm{B}_2^2$ is uniformly positive. Fix a local orthonormal frame $e_1,\dots,e_m$ of~$TM_2$. By inserting the definition of~$\mathrm{B}_2$ and using the local expression of the Dirac operator, we obtain
\begin{align} 
  \mathrm{B}_2^2 
  &= \Dirac_2^2 + \lambda \Dirac_2\overline{c}\brackets[\big]{\pr_1^*\xi_2}+\lambda \overline{c}\brackets[\big]{\pr_1^*\xi_2} \Dirac_2+ \lambda^2 \overline{c}\brackets[\big]{\pr_1^*\xi_2}^2 \\
  &= \Dirac_2^2+ \lambda \sum_{i=1}^m \brackets[\Big]{c\brackets{e_i}\nabla_{e_i}\overline{c}\brackets[\big]{\pr_1^*\xi_2}+\overline{c}\brackets[\big]{\pr_1^*\xi_2}c\brackets{e_i}\nabla_{e_i}}+\lambda^2 \norm[\big]{\pr_1^*\xi_2}^2_{\placeholder} \\
  &= \Dirac_2^2+\lambda \sum_{i=1}^m c\brackets{e_i}\overline{c}\brackets{\nabla_{e_i}\pr_1^*\xi_2}+\lambda^2 \norm[\big]{\pr_1^*\xi_2}^2_{\placeholder}.
\end{align}
Here we used in the last step the compatibility of the connection with the Clifford multiplication and the anti-commutativity of~$c$ and~$\overline{c}$. Since $c\brackets{e_i}\overline{c}\brackets[\big]{\nabla_{e_i}\pr_1^*\xi_2}$ is a self-adjoint operator on $\Ltwo{M_2,W_2}$ for all $i\in \set{1,\dots,m}$, we obtain with $C_1\coloneqq \max_{i\in \set{1,\dots,m}}\opnorm{c\brackets{e_i}\overline{c}\brackets[\big]{\nabla_{e_i}\pr_1^*\xi_2}}$ and $C_2\coloneqq \min_{p\in M_2}\norm[\big]{\pr_1^*\xi_2}^2_p$
\begin{equation} \label{EstimateLtwoB2}
  \begin{split}
  \Ltwoscalarproduct[\big]{\mathrm{B}_2^2u}{u}
  &= \Ltwoabs{\Dirac_2u}^2+\lambda \sum_{i=1}^m \underbrace{\Ltwoscalarproduct[\big]{c\brackets{e_i}\overline{c}\brackets[\big]{\nabla_{e_i}\pr_1^*\xi_2} u}{u}}_{\geq -C_1 \Ltwoabs{u}^2} +\lambda^2 \underbrace{\Ltwoscalarproduct[\big]{\norm[\big]{\pr_1^*\xi_2}^2u}{u}}_{\geq C_2 \Ltwoabs{u}^2}\: \\
  &\geq \brackets[\Big]{C_2\lambda^2-C_1m\lambda} \Ltwoabs{u}^2
  \end{split}
\end{equation}
for all $u\in \Hsp{2}{M_2,W_2}$ and all~$\lambda>0$. The vector field~$\xi_2$ does nowhere vanish by construction, hence~$C_2>0$. Therefore, we can choose~$\lambda$ sufficiently big such that by \cref{EstimateLtwoB2} the operator $\mathrm{B}_2^2$ is uniformly positive. We obtain by \cref{DiracInvertible} that the index of~$\mathrm{B}_2$ vanishes. Since the vector field obtained from~$\xi_1$ and~$\xi_2$ by following the cut-and-paste procedure vanishes nowhere, the same calculations yield $\ind \mathrm{B}_3=0$ for~$\lambda$ sufficiently big. This proves part~(2). \medbreak
\textbf{Ad(3)} Let~$\bar{\pr}_1^i$ and~$\bar{\pr}_2^i$ be the projections of $\R^n\times M_{p_i}$ onto the first and second factor, respectively, for all $i\in \set{1,\dots,l}$ and $\bar{\pr}_1$ the union of all $\bar{\pr}_1^i$. The $\ComplexCl_{m,n}\tensgr \A$-linear Dirac bundle $W_4\to M_4$ is given by
\begin{equation} \label{DefinitonW4}
  W_4=\bigsqcup_{i=1}^l \SpinBdl \brackets{\R^n\times M_{p_i}}\tensgr \brackets[\big]{\bar{\pr}_1^i}^*\overline{\SpinBdl} \R^n \tensgr \brackets[\big]{\bar{\pr}_2^i}^*E\vert_{M_{p_i}} \to M_4\cong\bigsqcup_{i=1}^l \R^n \times M_{p_i}
\end{equation}
with an appropriate metric on~$\R^n$. Here the isomorphism $M_4\cong\bigsqcup_{i=1}^l \R^n \times M_{p_i}$ is constructed from~$\psi_i$, and we write $\overline{\SpinBdl}\R^n$ for the $\ComplexCl_{0,n}$-linear spinor bundle of~$T\R^n$. Define a graded Real $\ComplexCl_{m,n}\tensgr \A$-linear Dirac bundle $\widetilde{W}_4\to M_4$ by using the Euclidean metric on~$\R^n$ and the product metric on~$\R^n\times M_{p_i}$. Denote its induced Dirac operator by~$\widetilde{\Dirac}_4$. Then the index of~$\mathrm{B}_4$ matches with the index of $\widetilde{\Dirac}_4+\lambda \overline{c}\brackets{\xi_4}$. Let~$\bar{\xi}$ be the vector field on $\bigsqcup_{i=1}^l \R^n$ defined via
\mapdefoneline{\bar{\xi}^i}{\R^n}{T\R^n}{x}{\diag\brackets{\lambda_{i1},\dots,\lambda_{in}}x}
for $i\in \set{1,\dots,l}$. Since the vector fields $\bar{\pr}_1^*\bar{\xi}$ and~$\xi_4$ match by construction outside a compact set, the operator family 
\begin{equation}
  \widetilde{\mathrm{B}}_4^t=\widetilde{\Dirac_4}+\lambda \bar{c}\brackets{\xi_4}+t\brackets{\lambda \bar{c}\brackets{\bar{\pr}_1^*\bar{\xi}}-\lambda \bar{c}\brackets{\xi_4}}
\end{equation}
is uniformly positive at infinity for all $t\in \sqbrackets{0,1}$. We obtain by \cite[Proposition 4.1]{Cecchini2020b} that the indices of $\widetilde{\mathrm{B}}_4^0$ and $\widetilde{\mathrm{B}}_4^1$ matches, hence the index of~$\mathrm{B}_4$ matches with the index of
\begin{equation}
  \widetilde{\mathrm{B}}_4=\widetilde{\Dirac}_4+\lambda \bar{c}\brackets{\bar{\pr}_1^*\bar{\xi}}.
\end{equation}
We use $\SpinBdl \brackets{\R^n\times M_{p_i}}\cong \SpinBdl \R^n\boxgr \SpinBdl M_{p_i}= \brackets[\big]{\bar{\pr}_1^i}^* \SpinBdl \R^n \tensgr \brackets[\big]{\bar{\pr}_2^i}^* \SpinBdl M_{p_i}$ 
and rewrite the bundle~$\widetilde{W}_4$ into
\begin{equation}
  \widetilde{W}_4 \cong \bigsqcup_{i=1}^l \brackets{\SpinBdl \R^n\tensgr \overline{\SpinBdl}\R^n}\boxgr \brackets{\SpinBdl M_{p_i}\tensgr E\vert_{M_{p_i}}} \to \bigsqcup_{i=1}^l \R^n\times M_{p_i}.
\end{equation}
We finally obtain, by sum and product formulas for the higher index,
\begineqarray \label{BigIndexCalculationEndOfProof}
  \ind \mathrm{B}_4 
  &=& \ind \brackets{\widetilde{\Dirac}_4+\lambda \bar{c}\brackets{\bar{\pr}_1^*\bar{\xi}}} \\
  &=& \sum_{i=1}^l \prod_{j=1}^n \underbrace{\ind \brackets{e_1\partial_x+\lambda \lambda_{ij}x\overline{e}_1}}_{\overset{(\star)}{=}\sgn\brackets{\lambda_{ij}}\in \Z\cong \KO_0\brackets{\ComplexCl_{1,1}}}\cdot \underbrace{\ind \brackets[\Big]{\Dirac_{SM_{p_i}\tensgr E\vert_{M_{p_i}}}}}_{\overset{(\star\star)}{=}\ind \brackets[\Big]{\Dirac_{SM_p\tensgr E\vert_{M_p}}}} \\
  &\overset{\text{(\ref{EulerCharPoincareHopf})}}{=}& \chi\brackets{N}\cdot \ind \brackets[\Big]{\Dirac_{SM_p\tensgr E\vert_{M_p}}}.
\endeqarray
Here the $"\cdot"$ denotes the K-theory product
\begin{equation}
  \KO_0\brackets{\ComplexCl_{n,n}}\times \KO_0\brackets{\ComplexCl_{k,0}\tensgr \A} \to \KO_0\brackets{\ComplexCl_{n+k,n}\tensgr \A}
\end{equation}
and the operator $e_1\partial_x+\lambda \lambda_{ij}x\overline{e}_1$ is considered as an operator from $\Hsp{1}{\R,\ComplexCl_{1,1}}$ to $\Ltwo{\R,\ComplexCl_{1,1}}$, where~$e_1,\overline{e}_1$ are generators of~$\ComplexCl_{1,1}$. Equation~$(\star\star)$ holds by the bordism invariance of the index. It remains to show equation~$(\star)$. We have an isomorphism of graded Real $\Cstar$-algebras $\rho\colon \ComplexCl_{1,1}\to \Mat_2\brackets{\C}\cong \Kom_{\C}\brackets{\C^2}$ given by
\renewcommand{\arraystretch}{1} \begin{equation}
  1\Mapsto \begin{pmatrix} 1&0\\0&1 \end{pmatrix}, \quad
  e_1\Mapsto \begin{pmatrix} 0&1\\-1&0 \end{pmatrix}, \quad
  \overline{e}_1\Mapsto \begin{pmatrix} 0&1\\1&0 \end{pmatrix}, \quad
  e_1\overline{e}_1\Mapsto \begin{pmatrix} 1&0\\0&-1 \end{pmatrix}, 
\renewcommand{\arraystretch}{2}\end{equation}
hence $\ComplexCl_{1,1}$ is strongly Morita equivalent to~$\C$. Here $\C$ is considered as a trivially graded $\Cstar$-algebra equipped with its natural Real structure, and $\C^2$ is the graded Real Hilbert space with $\brackets{1,0}^T$ even and $\brackets{0,1}^T$ odd. This implements as in \cite[Proposition 12.17]{Roe} an isomorphism on Real K-theory, which takes in the Fredholm picture the form
\mapdefonelinenoname
  {\KO_0\brackets{\ComplexCl_{1,1}}}
  {\KO_0\brackets{\C}}
  {\sqbrackets{\brackets{F,Q}}}
  {\sqbrackets{\brackets{F\tens id_{\C^2},Q\tens_{\rho}\C^2}}.}
Together with the isomorphism $\ComplexCl_{1,1}\tens_{\rho}\C^2\cong \C^2$ and $\KO_0\brackets{\C}\cong \Z$, we obtain
\begineqarrayst 
  \ind\brackets{e_1\partial_x+\lambda \lambda_{ij}x\overline{e}_1} &&\quad \in \KO_0\brackets{\ComplexCl_{1,1}} \\
  \quad\quad \widehat{=} \ind \renewcommand{\arraystretch}{1}\begin{pmatrix} 0&-\partial_x+\lambda \lambda_{ij}x\\ \partial_x+\lambda \lambda_{ij}x&0 \end{pmatrix}\renewcommand{\arraystretch}{2} &&\quad \in \KO_0\brackets{\C} \\
  \quad\quad \widehat{=}\dim\ker\brackets{\partial_x+\lambda \lambda_{ij}x}-\dim\ker\brackets{-\partial_x+\lambda \lambda_{ij}x}=\sgn\brackets{\lambda_{ij}} &&\quad \in \Z.
\endeqarray
Here we consider $\renewcommand{\arraystretch}{0.6}\begin{pmatrix} 0&-\partial_x+\lambda \lambda_{ij}x\\ \partial_x+\lambda \lambda_{ij}x&0 \end{pmatrix}\renewcommand{\arraystretch}{2}$ as an operator from $\Hsp{1}{\R,\C^2}$ to $\Ltwo{\R,\C^2}$, and we use in the last step 
\begin{equation}
  \ker\brackets{\partial_x+\lambda \lambda_{ij}x}=
  \begin{cases}
    \angles{e^{-\frac{1}{2}\lambda \lambda_{ij}x^2}} &\lambda_{ij}>0 \\
    0 &\lambda_{ij}<0
  \end{cases},\;
  \ker\brackets{-\partial_x+\lambda \lambda_{ij}x}=
  \begin{cases}
    0 &\lambda_{ij}>0 \\
    \angles{e^{\frac{1}{2}\lambda \lambda_{ij}x^2}} &\lambda_{ij}<0.
  \end{cases}
\end{equation}
Note that $e^{-\frac{1}{2}\lambda \lambda_{ij}x^2}$ is for $\lambda_{ij}<0$ not in $\Hsp{1}{\R,\C^2}$, hence not in the kernel of $\partial_x+\lambda \lambda_{ij}x$. This proves~$(\star)$ in \cref{BigIndexCalculationEndOfProof}, hence part~(3) is shown and the index theorem is proved. 
\end{proof} 
We combine \cref{RigidityAlmostHarmonicSpinorThm} with the index formula in \cref{IndexThm} and obtain the following extremality and rigidity statement:  
\begin{thm}[Extremality and Rigidity] \label{RigidityThmHilbertBundleE}
  Let $f\colon M\to N$ be a smooth area non-increasing spin map between two closed connected Riemannian manifolds of dimension~$n+k$ and~$n$, respectively. Suppose the curvature operator of~$N$ is non-negative and the following holds:
  \begin{enumerate}
    \item[(\mylabel{T}{T})] 
      There exists a unital Real $\Cstar$-algebra~$\A$ and a Real flat bundle~$E\to M$ of finitely generated projective Hilbert $\A$-modules such that for a regular value~$p$ of~$f$
      \begin{equation} \label{TopologicalConditionHigherDegree}
        \chi\brackets{N} \cdot \ind \brackets[\Big]{\Dirac_{\SpinBdl M_p\tensgr E\vert_{M_p}}} \neq 0 \in \KO_k\brackets{\A}.
      \end{equation}
  \end{enumerate}
  Then $\scal_M\geq \scal_N\op f$ on $M$ implies $\scal_M=\scal_N \op f$. If, moreover, $\scal_N>2\Ric_N>0$ (or~$f$ is distance non-increasing and $\Ric>0$), then $\scal_M\geq \scal_N\op f$ implies that~$f$ is a Riemannian submersion.  
\end{thm}
\begin{proof}
  Let $\DE$ be the twisted Dirac operator associated to the map~$f$ and the bundle~$E$ as explained before \cref{RigidityAlmostHarmonicSpinorThm}. Since the product of the Euler characteristic of~$M$ with the index of the twisted Dirac operator over the fiber of the point~$p$ does not vanish (\cref{TopologicalConditionHigherDegree}), we obtain by the index formula in \cref{IndexThm} that the index of $\DE$ does not vanish in $\KO_k \brackets{\A}$. By \cref{DiracInvertible} (2) the Dirac operator $\DE$ is not invertible, hence there exists a family of almost $\DE$-harmonic sections by \cref{AlmostHarmonicExistence}. The statements in the theorem follow from \cref{RigidityAlmostHarmonicSpinorThm}.  
\end{proof}
For certain types of manifolds, e.g.\ symmetric spaces of compact type, the estimate on the scalar and Ricci curvature are automatically fulfilled, hence the previous theorem can be simplified significantly. See the discussion in the end of Section~1 in \cite{Goette2000}.\par
We want to focus on the index-theoretical condition~(\ref{T}) and give in the next corollary three alternative versions of \cref{RigidityThmHilbertBundleE}. In the first version, we replace the topological condition~(\ref{T}) by a concrete condition involving the higher mapping degree. This seems to be less general, but \cref{WhyMishchenkoProp} shows that the higher mapping degree is the most general invariant for the purposes in \cref{RigidityThmHilbertBundleE}. In the second version, we replace the topological condition~(\ref{T}) on the higher mapping degree by a condition involving the Rosenberg index of the fiber over a regular value of~$f$. This provides a wider range of accessible applicability of the extremality and rigidity statement, since the Rosenberg index is studied for many classes of manifolds. For example, any enlargeable or area-enlargeable spin manifold has non-vanishing rational Rosenberg index. For this simplification, we have to require the existence of a \textit{retraction} of the fiber over a regular value~$p$ of~$f$ to the manifold~$M$ \textit{on the level of fundamental groups}. This is a homomorphism $r\colon \pi_1\brackets{M} \to \pi_1\brackets{M_p}$ satisfying $r\op \iota_*=id$, where~$\iota_*\colon\pi_1\brackets{M_p}\to \pi_1\brackets{M}$ is induced by the inclusion $\iota \colon M_p\xhookrightarrow{} M$. The third version contains the $\hat{A}$-degree of~$f$ and recovers the classical result by \textcite[Theorem 2.4]{Goette2000}. 
\begin{cor} \label{RigidityThmHigherDergeeRosenberg}
  \cref{RigidityThmHilbertBundleE} still holds if we replace the index-theoretical condition~(\ref{T}) by one of the following conditions: 
  \begin{enumerate}
    \item[(\mylabel{T1}{T1})] 
      The product of the Euler characteristic of~$N$ and the higher degree of~$f$ does not vanish in~$\KO_k\brackets[\big]{\Cstar{\pi_1\brackets{M}}}$. 
      \item[(\mylabel{T2}{T2})] 
      \begin{samepage} 
        There exists a regular value~$p$ of~$f$ such that $M_p$ is connected,
      \begin{equation} \label{IndexHypothesisinT2}
        \chi\brackets{N} \cdot \alpha\brackets{M_p} \neq 0 \in \KO_k\brackets{\Cstar \pi_1\brackets[\big]{M_p}},
      \end{equation}
      and there exists a retraction of~$M_p$ to the manifold~$M$ on the level of fundamental groups. 
      \item[(\mylabel{T3}{T3})] 
        The $\hat{A}$-degree of~$f$ and the Euler characteristic of~$N$ do not vanish. 
      \end{samepage}
  \end{enumerate}
  \end{cor}
\begin{proof}
  The first and third part follow by choosing in \cref{RigidityThmHilbertBundleE} the bundle~$E$ equal to the Mishchenko bundle of~$M$ and the trivial flat line bundle $\C\times M\to M$, respectively. The product of the Euler characteristic of~$N$ and the $\hat{A}$-degree of~$f$ matches with the index in \cref{TopologicalConditionHigherDegree} by \cref{AhatDegreeRewrittenEq}, the isomorphism $\KO_0\brackets{\C}\cong \Z$, and the fact that twisting with the trivial flat line bundle does not change the index of the induced Dirac operator. It remains to show the second part of the theorem. Let~$p$ be as in (\ref{T2}) and $r\colon \pi_1\brackets{M} \to \pi_1\brackets{M_p}$ a retraction. We obtain a representation
  \begin{equation}
    \rho \colon \pi_1\brackets{M} 
    \To{r} \pi_1\brackets{M_p}
    \To{} \Lin_{\Cstar\pi_1\brackets{M_p}}\brackets{\Cstar\pi_1\brackets{M_p}}
  \end{equation}
  by left multiplication, which induces a Real flat bundle of finitely generated projective Hilbert $\Cstar\pi_1\brackets{M_p}$-modules
  \begin{equation}
    E\coloneqq \widetilde{M}\times_{\rho}\Cstar\pi_1\brackets{M_p}.
  \end{equation} 
  Since~$r$ is a retraction, $\iota_*$ is injective, hence any connected component of the preimage of~$M_p$ under the projection $\widetilde{M}\to M$ is a universal cover of~$M_p$ \cite[][proof of 1B.11]{Hatcher2001}. We fix such a connected component and denote it by~$\widetilde{M}_p$. We obtain an isomorphism
  \begin{equation} \label{IsoRestrictionEtoL}
    E\vert_{M_p}
    =\widetilde{M}\times_{\rho} \Cstar \pi_1\brackets{M_p} \vert_{M_p} 
    \cong \widetilde{M}_p \times_{\pi_1\brackets{M_p}}\Cstar \pi_1\brackets{M_p}
    = \mathcal{L}\brackets{M_p}.
  \end{equation}
  This gives, by the definition of the Rosenberg index in \cref{DefRosenbergIndexEq}, 
  \begin{equation}
    0 \neq \chi\brackets{N} \cdot \alpha\brackets{M_p} 
    = \chi\brackets{N} \cdot \ind \brackets[\Big]{\Dirac_{\SpinBdl M_p\tensgr \mathcal{L}\brackets{M_p}}}
    =\chi\brackets{N} \cdot \ind \brackets[\Big]{\Dirac_{\SpinBdl M_p\tensgr E\vert_{M_p}}},  
  \end{equation} 
  hence the topological condition (\ref{T}) in \cref{RigidityThmHilbertBundleE} is fulfilled, and the corollary is proved. 
\end{proof}
\begin{rem}[Relation between topological conditions] \label{ConnectionTT1T2}
  By the considerations in the proof of \cref{RigidityThmHigherDergeeRosenberg}, the index formula in \cref{IndexThm} and \cref{WhyMishchenkoProp}, the three topological properties in \cref{RigidityThmHilbertBundleE} and \cref{RigidityThmHigherDergeeRosenberg} are connected as follows:
  \begin{equation}
    (\ref{T2}) \quad \Longrightarrow \quad (\ref{T}) \quad \Longleftrightarrow \quad (\ref{T1}) \quad \Longleftarrow \quad (\ref{T3})
  \end{equation}
\end{rem} 
It follows some examples, where the topological hypothesis~(\ref{T2}) of \cref{RigidityThmHigherDergeeRosenberg} is fulfilled. The examples~(\ref{ExampleTorus}) and~(\ref{ExExoticSpehere}) are explicit situations in which the theorem by \textcite[Theorem 2.4]{Goette2000} is not applicable, since the $\hat{A}$-degree vanishes. 
\begin{ex}[product] \label{ExampleProduct}
  Let~$F$ and~$N$ be closed connected manifolds and $\pr_1\colon N\times F \to N$ the projection to the first factor. Suppose the first and second Stiefel Whitney classes of~$F$ vanish. Then $\pr_1$ is a spin map, and since~$N$ is a retract of~$N\times F$, the topological condition~(\ref{T2}) in \cref{RigidityThmHigherDergeeRosenberg} reduces to 
  \begin{equation} \label{IndexEqautionProductExample}
    \chi\brackets{N} \cdot \alpha\brackets{F} \neq 0 \in \KO_k\brackets{\Cstar \pi_1\brackets[\big]{F}}.
  \end{equation}
  This holds for example if the Euler characteristic of~$N$ does not vanish and the fiber~$F$ has rationally non-vanishing Rosenberg index,
  e.g.:
  \begin{itemize}
    \item $F$ is enlargeable or area-enlargeable \cite[see][]{Hanke2007}.
    \item $F$ admits a metric of non-positive sectional curvature.
    \item 
      $F$ is aspherical
      and the strong Novikov conjecture holds for its fundamental group, i.e.\ the rational assembly map 
      \begin{equation}
        \nu\colon \KO_*\brackets{F}\otimes_\Z \Q \To{} \KO_*\brackets{\Cstar \pi_1\brackets{F}} \otimes_\Z \Q
      \end{equation}
      is injective \cite[][]{Rosenberg1983}. Note that for aspherical manifolds $F\cong B\pi_1\brackets{F}$. 
  \end{itemize}
  In the third example, the non-vanishing of the Rosenberg index in $\KO_0\brackets{\Cstar \pi_1\brackets{F}}\otimes_\Z \Q$ follows from the fact that the $\KO$-fundamental class of the closed spin manifold~$F$ is rationally non-zero and get mapped under the rational assembly map~$\nu$ onto $\alpha \brackets{M}$. The second example reduced to the first case, since any compact spin manifold that admits a metric of non-positive sectional curvature is enlargeable by the Cartan--Hadamard theorem \cite[Theorem 1.10]{Lee2018}. The first example can be verified by analyzing the proof of \textcite{Hanke2007} that any enlargeable or area-enlargeable spin manifold has non-vanishing Rosenberg index. They construct a unital $\Cstar$-algebra~$Q$ together with a group homomorphism 
  \begin{equation}
    \phi_*\colon \KO_0\brackets{\Cstar{\pi_1\brackets{F}}} \To{} \K_0\brackets{Q}
  \end{equation}
  such that the group $\K_0\brackets{Q}$ splits off a summand $\prod_{i\in \N} \Z\slash \bigoplus_{i\in \N} \Z$ and the component of $\phi_*\brackets{\alpha\brackets{F}}$ is non-zero in this summand. It follows that $\alpha \brackets{F}$ is rationally non-zero. 
\end{ex}
\begin{ex}[$n$-torus] \label{ExampleTorus}
  Since the $n$-torus is enlargeable, the rigidity statement in \cref{RigidityThmHilbertBundleE} applies to 
  \begin{equation}
    S^{2n} \times T^k \To{\pr_1} S^{2n} \quad \text{and} \quad \RP^{2n} \times T^k \To{\pr_1} \RP^{2n}.
  \end{equation} 
  The rigidity statement for the round sphere reads for $n\geq 2$ as follows: If $T^k \times S^{2n}$ is equipped with a metric $g$ such that $\pr_1$ is area non-increasing and $\scal_g\geq 2n\brackets{2n-1}$, then~$\pr_1$ is a Riemannian submersion with $\scal_g=2n\brackets{2n-1}$.
\end{ex}
\begin{ex}[exotic sphere] \label{ExExoticSpehere}
  Let $\Sigma^{8k+j}$ be an exotic sphere with $j\in \set{1,2}$ and non-vanishing Hitchin invariant \cite[see][]{Hitchin1974}, and let~$f\colon M\to N$ be a smooth area non-increasing spin map between two closed connected Riemannian manifolds. Then the condition~(\ref{T2}) in \cref{RigidityThmHigherDergeeRosenberg} holds if the Euler characteristic of~$N$ is odd and there exists a regular value~$p$ of~$f$ such that the fiber~$M_p$ is isomorphic to~$\Sigma^{8k+j}$. This follows from $\pi_1\brackets{\Sigma^{8k+j}}=0$ and $\chi\brackets{N}\cdot \alpha\brackets{\Sigma^{8k+j}}\neq 0\in \Z_2$. The latter holds by Bott periodicity $\KO_{8k+j}\brackets{\C}\cong \Z_2$ \cite{Atiyah1964} and the fact that the Rosenberg index of a simply connected manifold matches with the Hitchin invariant. Maps of this type are for example 
  \begin{equation}
    \RP^{2n} \times \Sigma^{8k+j} \To{\pr_1} \RP^{2n} \quad \text{and} \quad \CP^{2n} \times \Sigma^{8k+j} \To{\pr_1} \CP^{2n}.
  \end{equation}
  Here the manifolds $\RP^{2n}$ and $\CP^{2n}$ are equipped with metrics of non-negative curvature operator and $\scal>2\Ric>0$. Such metrics are given for example by the standard metric on~$\RP^{2n}$ and the Fubini–Study metric on~$\CP^{2n}$.
\end{ex}
\begin{rem}[Rigidity for non-orientable manifolds] \label{RemarkOrientable}
  We assume in \cref{RigidityThmHilbertBundleE} that the map~$f$ is orientable, but not that the involved manifolds are orientable. The orientability assumption in the classical statement by \textcite[][Theorem 2.4]{Goette2000} does not restrict the applicability of the classical theorem, since one can always apply it to the lift of the map to the oriented double coverings. This trick also works in our case for the map $\pr_1\colon \RP^{2n}\times T^k\to \RP^{2n}$ considered in \cref{ExampleProduct}, but not for the map $\RP^{2n} \times \Sigma^{8k+j} \To{\pr_1} \RP^{2n}$ in \cref{ExExoticSpehere}. The reason is that the lift $S^{2n} \times \Sigma^{8k+j} \To{\pr_1} S^{2n}$ does not fulfill the topological condition~(\ref{T1}), since the Euler characteristic of~$S^{2n}$ is two and the Hitchin invariant of~$\Sigma^{8k+j}$ is two torsion. This shows that including non-orientable manifolds in \cref{RigidityThmHilbertBundleE} extends the applicability of the theorem.
\end{rem}
In \cref{ExampleProduct}, we considered projections from a product of two manifold onto one of its factors. A generalization of this class of examples to possible non-trivial fiber bundles as in \cref{ExExoticSpehere} requires more topological conditions on the target manifold~$N$. The reason is that we have to ensure, additionally to the index-theoretical condition on the Euler characteristic and the Rosenberg index of the fiber, that the condition on the fundamental groups in~(\ref{T2}) is fulfilled. It is sufficient for the existence of a retraction from~$M_p$ to~$M$ on the level of fundamental groups that the target manifold is 2-connected (\cref{CorRigidityFiberBundles}).\par 
All simply connected closed Riemannian manifolds of non-negative curvature operator are classified. See for instance \cite[Theorem 1.13]{Wilking2007}.
This leads, together with the classification of irreducible symmetric spaces by \textcite{Cartan1926}, to \cref{LemClassificationN}. In \cref{CorRigidityFiberBundles}, we state the rigidity statement for fiber bundles over these manifolds.
\begin{lem} \label{LemClassificationN}
  Any factor of the de Rham decomposition of a closed 2-connected Riemannian manifold with non-negative curvature operator and non-vanishing Euler characteristic is isometric to a sphere~$S^{2n}$, $n\geq 2$, equipped with a metric of non-negative curvature operator, a quaternionic Grassmannian or the Cayley projective plane. 
\end{lem}
The \textit{quaternionic Grassmannian} is the $4pq$-dimensional symmetric space C~\textrm{II}, given by $\Sp\brackets{p+q}\slash \Sp\brackets{p}\times \Sp\brackets{q}$, and the \textit{Cayley projective plane} is the 16-dimensional symmetric space F~\textrm{II}, given by $F_4\slash \Spin\brackets{9}$. Here we denote by $\Sp\brackets{k}$ the compact symplectic group and by $F4$ the simply connected Lie group associated to the Lie algebra~$\mathfrak{f}_4$ \cite[see][]{Yokota2009}. Both manifolds are equipped with their unique (up to scaling) symmetric Riemannian metrics. 
\begin{proof}
  Let~$N$ be a factor in the de Rham decomposition of a 2-connected closed manifold with non-zero Euler characteristic and non-negative curvature operator. We obtain by \cite[Theorem 1.13]{Wilking2007} that~$N$ is isometric to one of the following Riemannian manifolds:
  \begin{itemize}
    \item Sphere $S^{2n}$, $n\geq 2$, with non-negative curvature operator.
    \item Compact irreducible 2-connected symmetric space with non-zero Euler characteristic.
  \end{itemize}
  Any irreducible symmetric space of compact type is listed in \cite[][Table~\textrm{V} on page 518]{Helgason2001} (type~\textrm{I}) or isometric to a compact connected simple Lie group equipped with a two-sided invariant metric
  (type~\textrm{II}). Since any Lie group admits a nowhere vanishing vector field, its Euler characteristic vanishes. It remains to rule out all symmetric spaces in \cite[][Table~\textrm{V} on page 518]{Helgason2001}, except the quaternionic Grassmannian and the Cayley projective plane by checking whether they are 2-connected and have non-zero Euler characteristic. Therefore, we use the results about the exceptional Lie groups~$E6$, $E7$, $E8$, $F4$ and $G2$ given in \cite{Yokota2009} and the fact that for any symmetric space~$G\slash K$ of compact type, with~$G$ a connected simple Lie group and~$K$ a compact connected subgroup, the following holds:
  \begin{enumerate}
    \item 
      The Euler characteristic of~$G\slash K$ is non-zero if and only if $\rk{G}=\rk{K}$.
    \item 
      The symmetric space $G\slash K$ is 2-connected if and only if the inclusion $\iota\colon K\xhookrightarrow{} G$ induces an isomorphism $\iota_*\colon \pi_1\brackets{K}\to \pi_1\brackets{G}$.
  \end{enumerate}
  The second part holds by $\pi_2\brackets{G}=0$ and the long exact sequence of homotopy groups
  \begin{equation}
    0 \To{} \pi_2\brackets{G\slash K} \To{} \pi_1\brackets{K} \To{i_*} \pi_1\brackets{G} \To{} \pi_1\brackets{G\slash K} \To{} 0. \qedhere \eqno \qed
  \end{equation}  
\end{proof}
\begin{cor} \label{CorRigidityFiberBundles}
  Let~$N$ be a product of Riemannian manifolds of the following form:
  \begin{itemize}
    \item Sphere~$S^{2n}$, $n\geq 2$, equipped with a metric satisfying $\mathcal{R}_{S^n}\geq 0$ and $\scal_{S^n}>2\Ric_{S^n}>0$.
    \item Quaternionic Grassmannian
    \item Cayley projective plane
  \end{itemize}
  Let $f\colon M\to N$ be a fiber bundle, where the total space~$M$ is a closed connected spin manifold and the typical fiber~$F$ satisfies
  \begin{equation} \label{IndexTheoreticalConditionFiberBundle}
    \chi\brackets{N}\cdot\alpha\brackets{F}\neq 0 \in \KO_{\dim \brackets{F}}\brackets{\Cstar \pi_1\brackets{F}}.
  \end{equation}
  Then~$f$ is a Riemannian submersion with $\scal_M=\scal_N \op f$ if the map~$f$ is area non-increasing and $\scal_M\geq \scal_N\op f$ holds. Equation (\ref{IndexTheoreticalConditionFiberBundle}) holds if the typical fiber~$F$ is area-enlargeable, or~$F$ is aspherical and the strong Novikov conjecture holds for its fundamental group.
\end{cor}
\begin{proof}
  The manifold~$N$ is 2-connected, has non-zero Euler characteristic and has non-negative curvature operator by \cref{LemClassificationN}. Furthermore, it satisfies $\scal_{S^n}>2\Ric_{S^n}>0$ by the discussion about locally symmetric spaces of compact type in \cite[end of section 1c]{Goette2000}. Let~$f\colon M\to N$ be a fiber bundle with typical fiber~$F$ as in the theorem. Since~$N$ is 2-connected, it is spin, hence~$f$ is a spin map. From the long exact homotopy sequence of fiber bundles 
  \begin{equation}
    \dots \To{} \pi_2\brackets{N} \To{} \pi_1\brackets{F} \To{\iota_*} \pi_1\brackets{M} \To{} \pi_1\brackets{N} \To{} \dots
  \end{equation}
  together with the 2-connectedness of~$N$, we obtain that $\iota_*\colon \pi_1\brackets{F} \to \pi_1\brackets{M}$ is an isomorphism. The condition~(\ref{T2}) in \cref{RigidityThmHigherDergeeRosenberg} is fulfilled, hence the rigidity statement holds. The second part of the corollary follows from the same considerations as in \cref{ExampleProduct} about area-enlargeable and aspherical manifolds.
\end{proof}
%
%
\printbibliography
\end{document}